%% file: main.tex
\documentclass[11pt]{article}
\usepackage[margin=1in]{geometry}
\usepackage{amsmath,amssymb,amsthm}
\usepackage{float}
\usepackage{hyperref}
\usepackage{microtype}
\usepackage{tikz}
\hypersetup{colorlinks=true,linkcolor=blue,citecolor=blue,urlcolor=blue}
\newtheorem{theorem}{Theorem}
\newtheorem{lemma}[theorem]{Lemma}
\newtheorem{proposition}[theorem]{Proposition}
\theoremstyle{definition}
\newtheorem{definition}{Definition}[section]
\theoremstyle{remark}
\newtheorem*{remark}{Remark}
\theoremstyle{plain}
\newcommand{\E}{\mathbb E}
\newcommand{\Prob}{\mathbb P}
\newcommand{\F}{\mathbb F}
\newcommand{\R}{\mathbb R}
\newcommand{\one}{\mathbf 1}
\DeclareMathOperator{\tr}{tr}

\DeclareMathOperator{\diag}{diag}
\DeclareMathOperator{\cum}{cum}
\newcommand{\norm}[1]{\left\lVert#1\right\rVert_2}
\title{Subspace embeddings with the rerandomized SRHT}
\author{Yuning Yang\thanks{School of Mathematics, Guangxi University, Nanning 530004, China. Email: yyang@gxu.edu.cn}}
\date{}
\begin{document}
\maketitle

\begin{abstract}
This work studies subspace embeddings obtained by two normalized real Walsh
transforms, two independent sign diagonals, and uniform coordinate
sampling without replacement. The main result shows that the prescribed sample size
$k=\min\{n,\lceil Cr/\varepsilon^2\rceil\}$, for a universal constant $C$,
suffices to preserve all squared norms on each fixed $r$-dimensional
subspace within $1\pm\varepsilon$ with probability at least $0.99$.
The result holds for every ambient Walsh dimension and all ranks,
and answers Problem TR-01 in the Open Problems in Numerical Linear
Algebra repository.
The proof controls joint entry cumulants of the transformed projection
through connected graph contractions and Walsh character identities.
These estimates then bound the expected trace of even powers of a product of
centered projections. A two-projection decomposition converts this
estimate into control of both spectral edges. Bernoulli sampling at arbitrary
densities and a deterministic upper bound near full sampling yield the
prescribed number of coordinates.
\end{abstract}

\section{Introduction and main result}\label{sec:introduction}

An oblivious subspace embedding reduces dimension while approximately
preserving all norms on a fixed subspace. Its distribution is chosen
without knowing that subspace. Such maps provide a basic tool for
randomized least-squares and low-rank approximation algorithms
\cite{hmt,boutsidis_gittens}. For an $r$-dimensional subspace and fixed
failure probability, Gaussian maps achieve squared-norm error
$\varepsilon$ with $O(r/\varepsilon^2)$ output coordinates
\cite[Section~1]{nelson_nguyen}.
Fast structured maps seek the same geometric guarantee with less
work when applying the map.

The subsampled randomized Hadamard transform (SRHT) applies random
signs, a normalized Hadamard transform, and coordinate sampling.
The fast Walsh--Hadamard transform permits multiplication in
$O(n\log(2n))$ operations, compared with $O(nk)$ for a dense
$n$-by-$k$ map. Its one-round construction has a logarithmic
obstruction: the coupon-collector construction of
Halko, Martinsson, and Tropp for the Fourier transform
\cite[Remark~11.2]{hmt}, in Tropp's Walsh formulation
\cite[Section~3.3]{tropp_srht}, uses an $r$-dimensional subspace in ambient dimension $r^2$
and requires a sample of order $r\log r$ even for injectivity.
The example concerns the traditional one-round law and motivates
asking what a further randomization can accomplish.

Problem TR-01 in the Open Problems in Numerical Linear Algebra repository
\cite{tr01} asks whether a second independent sign diagonal and Hadamard
transform yield dimension proportional to $r/\varepsilon^2$ at a prescribed
sample size. This work answers that question affirmatively for uniform
sampling without replacement. The normalization and randomness used
throughout the paper are specified first.

Let \(n=2^m\), \(m\ge0\), and identify the coordinate set with
\(K=(\F_2)^m\). The real normalized Walsh matrix is
\[
 H_{ab}=n^{-1/2}(-1)^{\langle a,b\rangle},\qquad H^T=H,\quad H^2=I_n.
\]
Let \(D_1,D_2\) have mutually independent uniform Rademacher diagonal
entries. Independently let \(J\) be a uniformly chosen \(k\)-element
subset of \(K\), and let \(S_J\) collect its coordinate vectors. Put
\[
 \Omega_{n,k}=\sqrt{\frac nk}\,D_1HD_2HS_J.
\]
For a deterministic \(n\times r\) frame \(V\), \(V^TV=I_r\), write
\[
 X=HD_2HD_1V,\qquad P=XX^T,\qquad \delta=\frac rn.
\]
Thus \(X^TX=I_r\), and the Gram matrix in question is
\begin{equation}\label{eq:sampled-gram}
 V^T\Omega_{n,k}\Omega_{n,k}^TV
   =\frac nk X^T S_JS_J^T X.
\end{equation}
Here and below $\norm{\cdot}$ denotes the Euclidean operator norm.

\begin{theorem}\label{thm:main}
There is a universal constant \(C\ge1\) such that for every power of
two \(n\), every \(1\le r\le n\), and every \(0<\varepsilon<1\), the
prescribed integer
\begin{equation}\label{eq:prescribed-width}
 k=\min\{n,\lceil Cr/\varepsilon^2\rceil\}
\end{equation}
satisfies
\begin{equation}\label{eq:main-bound}
 \sup_{V^TV=I_r}
 \Prob\!\left\{\norm{V^T\Omega_{n,k}\Omega_{n,k}^TV-I_r}
                      >\varepsilon\right\}\le0.01.
\end{equation}
\end{theorem}

TR-01 states the prescribed-width question for $0<\varepsilon<1/2$;
Theorem~\ref{thm:main} proves the slightly stronger range $0<\varepsilon<1$.

\begin{remark}[Relation to the workshop formulation]
Problem~5.6 in the Simons workshop collection
\cite[Definitions~5.1 and~5.3, Problem~5.6]{workshop} asks whether the
rerandomized SRHT is an oblivious subspace embedding with dimension
$k=O(r/\varepsilon^2)$. TR-01 makes the width
$k=\min\{n,\lceil Cr/\varepsilon^2\rceil\}$, normalization, sampling
without replacement, and success probability $0.99$ explicit.
The workshop assumption $n=\Omega(\log r)$ imposes no additional
restriction here, since $\log r\le r\le n$.

An earlier version of this work answered the workshop question by
establishing the existence of a suitable width $k=O(r/\varepsilon^2)$.
The proof and its Lean~4 formalization are available in the
\href{https://github.com/yuningyang19/rerandSRHT_prob_5_6_simons_workshop/tree/5fb87c321a65baeaac09c79458a7356944647116}{earlier companion repository}.
Theorem~\ref{thm:main} strengthens that result by establishing the
guarantee at the prescribed width
$k=\min\{n,\lceil Cr/\varepsilon^2\rceil\}$, thereby resolving TR-01.
\end{remark}

For a fixed frame $V$, the event complementary to that in
\eqref{eq:main-bound} says precisely that
\[
 (1-\varepsilon)\|z\|_2^2
 \le \|\Omega_{n,k}^T Vz\|_2^2
 \le (1+\varepsilon)\|z\|_2^2
 \qquad\text{for every }z\in\R^r.
\]
The supremum in \eqref{eq:main-bound} is outside the probability:
the frame is fixed before both sign diagonals and the sample, while
the prescribed width depends only on $n,r,\varepsilon$ and satisfies
$r\le k\le n$. The proof controls the Gram matrix directly at this
width. Full sampling gives the Gram matrix $I_r$ exactly.

\subsection*{Proof overview}

The fixed-size sample is first replaced by a Bernoulli coordinate
projection $E$, independent of $P=XX^T$, whose diagonal entries
have mean $0<\theta<1$. The independent selectors make it possible to
expand the trace by grouping equal coordinate indices. The normalized
Gram error is $\theta^{-1}X^T(E-\theta I_n)X$.
A direct trace expansion therefore involves entries of $P$.
Since $\E P=\delta I_n$, with $\delta=r/n$, centering $P$ as well gives
\[
 R=P-\delta I_n,\qquad W=E-\theta I_n,\qquad A=RW.
\]
This second centering removes every first-order entry cumulant,
including the diagonal ones. Consequently, the cumulant expansion
contains only blocks of at least two entries, a restriction needed
for the later counts.

Centering both factors leaves a product $A$ that is generally
nonsymmetric, and its trace at an even order $2p$ can be negative.
Lemma~\ref{lem:transfer} deals with this difficulty deterministically.
The two-projection decomposition expresses the compressed Gram
eigenvalues through scalar and two-dimensional invariant blocks.
In each two-dimensional block of $A$, the two roots have a fixed
product, while their sum changes with the Gram eigenvalue.
This gives the lower bound
$\tr A^{2p}\ge-2r[\delta(1-\delta)\theta(1-\theta)]^p$.
After adding this deterministic compensation, a deviation at either
spectral edge forces a large trace contribution, under the
hypotheses of that lemma. Thus it suffices to bound the signed
expectation $\E\tr A^{2p}$; no estimate of
$\E|\tr A^{2p}|$ is required.

The random part of the proof must control both the size and the
number of terms in this trace expansion. Section~\ref{sec:cumulants}
first bounds the joint entry cumulants. The product formula in
Lemma~\ref{lem:cumulant-identities} expands the two independent sign
layers over connected partitions. Each resulting graph sum contains
a rank-$r$ projection edge; factoring it through $\R^r$ allows
Lemma~\ref{lem:graph} to bound the sum by $r$.
Lemma~\ref{lem:cumulants} combines this estimate with Walsh character
symmetry, which makes a cumulant vanish unless its endpoint labels
satisfy a binary equation. Section~\ref{sec:trace-proof} then counts
the terms that can remain. Lemma~\ref{lem:count} bounds the equality
patterns of the sampling indices, Lemma~\ref{lem:binary-constraints}
counts labels satisfying the binary equations, and
Proposition~\ref{prop:entry-contribution} bounds the contribution of
one equality pattern. Each loss of an equality block removes a free
row label together with its selector factor, giving $(n\theta)^{-1}$.
Losing a cumulant block gives $r^{-1}$, and an independent binary
equation gives $n^{-1}$.
Keeping these factors in the count offsets its growth with the
moment order and yields Proposition~\ref{prop:trace}.

Combining that trace estimate with the deterministic transfer gives,
under the hypotheses of Lemma~\ref{cor:bernoulli},
\[
 \Prob\{\norm{\theta^{-1}X^TEX-I_r}>\eta\}
 \le (K_0^p+2r)\left(\frac{4r}{\kappa\eta^2}\right)^p,
 \qquad \kappa=n\theta,
\]
where $0<\eta<1$, $\kappa$ is the expected sample size, and $K_0$
is the universal constant in Proposition~\ref{prop:trace}.
When $\kappa$ is a sufficiently large constant multiple of
$r/\eta^2$, both terms on the right decay geometrically in $p$,
apart from the factor $r$. A moment order proportional to $\log r$
absorbs this factor without introducing an extra logarithm into
the sampling width. The rank condition in
Lemma~\ref{cor:bernoulli} restricts this argument to sufficiently
large $r$.

Section~\ref{sec:main-proof} completes the sampling and parameter
choices. Lemma~\ref{lem:sampling} couples a uniform fixed-size sample
between two Bernoulli samples when the upper density is below one;
set inclusion compares their Gram matrices, and binomial tails
control the sample counts. At higher sampling fractions, the
deterministic bound $(n/k)X^TS_JS_J^TX\preceq(n/k)I_r$ controls
the upper edge, while the lower Bernoulli comparison remains valid.
For the remaining ranks, Lemma~\ref{lem:small} averages over both
rerandomization and sampling to obtain a second-moment bound
$r(r+1)/k$. These ranks lie below a fixed universal threshold
$R_0$, so this bound is sufficient at the same prescribed width
after absorbing $R_0$ into $C$. Full sampling gives an exact
Gram matrix when the minimum in \eqref{eq:prescribed-width} equals $n$.
Figure~\ref{fig:proof-dependencies} records the dependencies among
the numbered results.

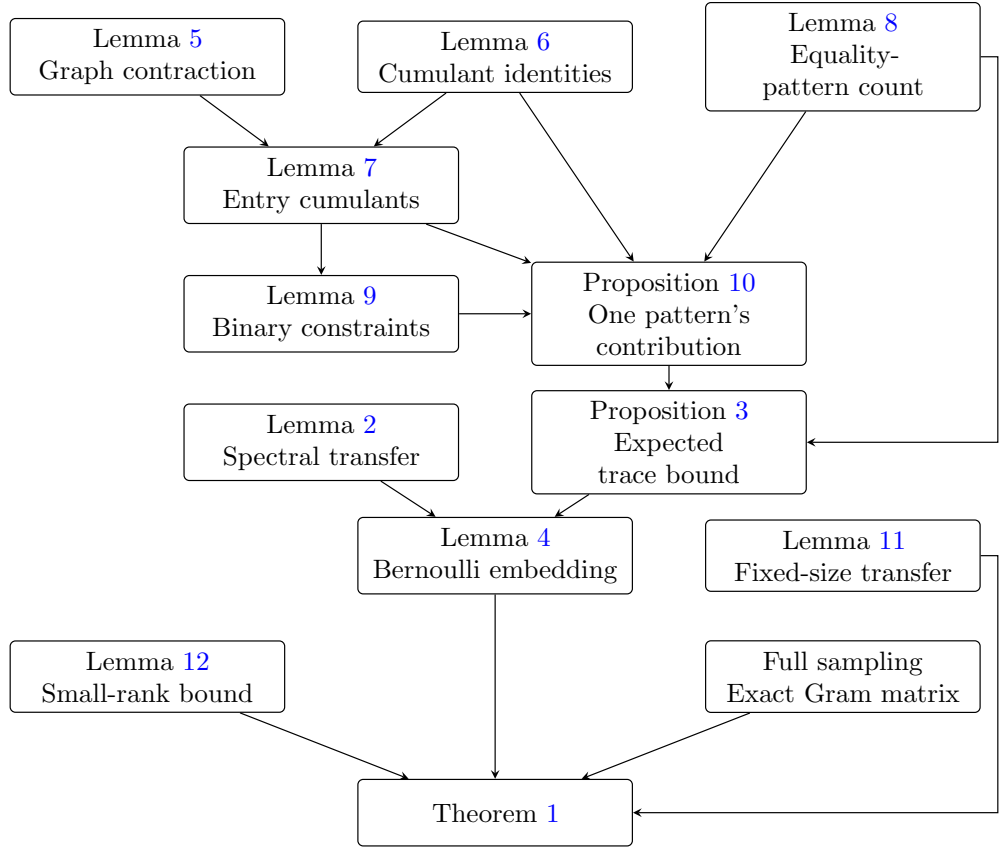
\begin{figure}[H]
\centering
\begin{tikzpicture}[
  >=stealth,
  result/.style={draw,rounded corners=2pt,align=center,
    text width=3.35cm,minimum height=0.88cm,inner sep=4pt,font=\small},
  every path/.style={line width=0.45pt}]
\node[result] (l5) at (-4.6,0)
  {Lemma~\ref{lem:graph}\\Graph contraction};
\node[result] (l6) at (0,0)
  {Lemma~\ref{lem:cumulant-identities}\\Cumulant identities};
\node[result] (l7) at (-2.3,-1.7)
  {Lemma~\ref{lem:cumulants}\\Entry cumulants};
\node[result] (l8) at (4.6,0)
  {Lemma~\ref{lem:count}\\Equality-pattern count};
\node[result] (l9) at (-2.3,-3.4)
  {Lemma~\ref{lem:binary-constraints}\\Binary constraints};
\node[result] (p10) at (2.3,-3.4)
  {Proposition~\ref{prop:entry-contribution}\\One pattern's contribution};
\node[result] (l2) at (-2.3,-5.1)
  {Lemma~\ref{lem:transfer}\\Spectral transfer};
\node[result] (p3) at (2.3,-5.1)
  {Proposition~\ref{prop:trace}\\Expected trace bound};
\node[result] (l4) at (0,-6.6)
  {Lemma~\ref{cor:bernoulli}\\Bernoulli embedding};
\node[result] (l11) at (4.6,-6.6)
  {Lemma~\ref{lem:sampling}\\Fixed-size transfer};
\node[result] (l12) at (-4.6,-8.2)
  {Lemma~\ref{lem:small}\\Small-rank bound};
\node[result] (full) at (4.6,-8.2)
  {Full sampling\\Exact Gram matrix};
\node[result] (main) at (0,-10.0)
  {Theorem~\ref{thm:main}};
\draw[->] (l5) -- (l7);
\draw[->] (l6) -- (l7);
\draw[->] (l7) -- (l9);
\draw[->] (l7) -- (p10);
\draw[->] (l6) -- (p10);
\draw[->] (l9) -- (p10);
\draw[->] (l8) -- (p10);
\draw[->] (l8.east) -- (6.65,0) -- (6.65,-5.1) -- (p3.east);
\draw[->] (p10) -- (p3);
\draw[->] (p3) -- (l4);
\draw[->] (l2) -- (l4);
\draw[->] (l4) -- (main);
\draw[->] (l11.east) -- (6.65,-6.6) -- (6.65,-10.0) -- (main.east);
\draw[->] (l12) -- (main);
\draw[->] (full) -- (main);
\end{tikzpicture}
\caption{Proof dependencies. An arrow points from a result to a
proof that uses it. Theorem~\ref{thm:main} combines the Bernoulli
estimate and fixed-size transfer for large ranks, the second-moment
bound for small ranks, and full sampling when needed. Definitions
and routine calculations are omitted.}
\label{fig:proof-dependencies}
\end{figure}

\subsection*{Relation to prior work}

The dimension benchmark concerns all oblivious linear maps, including
dense ones. For $0<\varepsilon<1/3$ and fixed failure probability
$0.01$, Nelson and Nguyen \cite[Theorem~6]{nelson_nguyen} prove the
lower bound
\[
 k=\Omega\!\left(\min\{n,r/\varepsilon^2\}\right).
\]
Their theorem uses relative norm error; squared-norm error
$\varepsilon$ implies that norm guarantee with the same $\varepsilon$.
The Gaussian upper bound, with a constant rescaling of its accuracy
parameter, attains the $r/\varepsilon^2$ order. The identity map
covers ambient-dimension saturation \cite[Section~1]{nelson_nguyen}.
This comparison explains both the dimension scale in Problem~5.6
and the necessity of retaining the cap by $n$.

Fast structured maps are particularly useful when applying a dense
sketch would dominate the matrix computation. Halko, Martinsson,
and Tropp \cite[Sections~4.6 and~11]{hmt} develop structured random
test matrices for approximate matrix decompositions.
Boutsidis and Gittens \cite[Theorems~2.1 and~3.1, Lemma~4.11]{boutsidis_gittens}
analyze SRHT low-rank approximation, least-squares regression and
matrix multiplication. For subspace embeddings, Tropp's analysis
\cite[Theorem~3.1]{tropp_srht} combines row-norm control with matrix
Chernoff bounds.

Multiple Hadamard/sign rounds have also been used for other objectives.
Andoni, Indyk, Laarhoven, Razenshteyn, and Schmidt
\cite[Section~3.1]{andoni_lsh} use a three-round transform as a fast
rotation in cross-polytope locality-sensitive hashing.
Choromanski et al.\ \cite[Lemma~1 and Theorems~5.1--5.3]{triplespin}
give distributional and hashing guarantees for three-round structures
within the TripleSpin framework. These are the multiround constructions
discussed after Problem~5.6 in \cite{workshop}; their conclusions concern
different observables from the uniformly sampled subspace Gram matrix.
For the same unsampled two-round rotation used here, Zilca and
Mendelson \cite[Equation~(1) and Theorem~4.1]{zilca_mendelson}
prove an $O(n^{-1/5})$ coordinate Kolmogorov bound relative to
a Haar-rotated vector, uniformly over fixed unit inputs and coordinate
indices. Their nonvanishing worst-case full-vector Wasserstein
discrepancy \cite[Theorems~5.1--5.2]{zilca_mendelson} addresses the
full output distribution. Theorem~\ref{thm:main} instead controls
the two spectral edges after coordinate sampling.

Other random-map laws achieve linear dependence on the subspace
dimension. Chenakkod, Derezi\'nski, Dong, and Rudelson
\cite[Theorem~4.5]{cddr_sparse} combine a Hadamard preprocessing step
with an independent sparse factor. With singular-value error
$\eta=\varepsilon/3$ and failure probability $0.01$, their theorem
gives dimension $O(r/\varepsilon^2)$ in the regime
$r>\log(200n)$, with its prescribed sparsity parameters.
The sparse factor takes the place of the second Walsh transform
and uniform coordinate subset in the present construction.

A related set of proof tools appears in the work of Chenakkod,
Derezi\'nski, and Dong \cite[Theorems~3 and~7]{cdd_icalp2025}.
For $r>10$, they study OSNAP embeddings and obtain dimension $O(r/\varepsilon^2)$
with near-optimal column sparsity when the inverse distortion and
inverse failure probability are at most polynomial in $r$.
Their analysis combines decoupling, cumulant expansions along
Gaussian interpolation, and trace inequalities
\cite[Section~4; full version, Sections~7.2--7.5]{cdd_icalp2025}.
In the fixed Walsh model considered here, connected graph
sums and character identities control the joint entry cumulants
directly. The subsequent trace is that of a generally nonsymmetric
product of centered projections; the two-projection calculation
provides the compensation required to recover both spectral edges.

Centered alternating traces of projections also appear in Kunisky's
analysis of the MANOVA law, a limiting eigenvalue distribution
for products of random projections
\cite[Theorems~1.5 and~1.12]{kunisky_manova}.
His moment recursions give sufficient conditions for convergence
in probability of the largest eigenvalue. The common starting point is the
centered trace; here the spectral transfer is a finite-dimensional
inequality with explicit compensation, and the moment estimate
must hold for the prescribed two-sign Walsh model.

V\'azquez-Becerra \cite[Section~4]{vazquez_becerra} combines
cumulant expansions, graph sums, and the bounds of Mingo and
Speicher to study fluctuations of traces, including ensembles
built from discrete Fourier matrices. These tools are closely
related to those in Section~\ref{sec:cumulants}. The distinction
needed here is control of fixed matrix entries, with explicit
dependence on cumulant order and on the rank of an arbitrary
fixed input projection. The remarks after
Lemmas~\ref{lem:transfer} and~\ref{lem:cumulants} give the
corresponding method comparisons.

The remainder is organized as follows.  Section~\ref{sec:reduction} gives the deterministic spectral reduction
and the Bernoulli embedding consequence of the trace estimate.
Section~\ref{sec:cumulants} proves the joint cumulant bounds, and
Section~\ref{sec:trace-proof} supplies the complete trace calculation.
Section~\ref{sec:main-proof} treats fixed-size sampling and all
remaining ranks and dimensions to prove Theorem~\ref{thm:main}.
Throughout, unadorned logarithms are natural, $\log_2$ is the binary
logarithm, and $\|\cdot\|_F$ denotes the Frobenius norm.

\section{From spectral deviations to a centered trace estimate}
\label{sec:reduction}

The argument begins with Bernoulli sampling. Let $E$ be a diagonal projection
whose entries are independent Bernoulli variables of mean $\theta$,
independent of the two sign diagonals. The matrix
$\theta^{-1}X^TEX$ is the corresponding normalized Gram matrix.
Since $X^TX=I_r$, its deviation is
\begin{equation}\label{eq:gram-centering}
 \theta^{-1}X^TEX-I_r
       =\theta^{-1}X^T(E-\theta I_n)X.
\end{equation}
Writing $W=E-\theta I_n$, cyclicity of the trace gives, for every
positive integer $j$,
\[
 \tr\bigl(X^TWX\bigr)^j=\tr(PW)^j,\qquad P=XX^T.
\]
Thus a direct moment expansion introduces entries of the
uncentered projection $P$. Lemma~\ref{lem:cumulants} below proves
that $\E P=\delta I_n$.
Subtracting this mean removes every first-order entry cumulant,
including those of diagonal entries. This is the reason for also
centering $P$: the entry-cumulant expansion will then contain only
blocks of size at least two. The resulting product need not be
symmetric, so a signed trace estimate requires a separate passage
back to the Gram eigenvalues.
Write $R=P-\delta I_n$, $W=E-\theta I_n$, and $A=RW$.
This section shows why controlling a signed even trace of $A$
suffices to bound deviations of that Gram matrix.

The first step is deterministic. For two orthogonal projections,
the compressed eigenvalues describe two-dimensional invariant
blocks. Centering the projections changes each block to a matrix
with a quadratic characteristic polynomial. Its real and nonreal
roots contribute differently to an even trace; the following lemma
keeps the possible negative contribution explicit.

\begin{lemma}\label{lem:transfer}
Let \(P=XX^T\) and \(E\) be any real orthogonal projections,
\(X^TX=I_r\). Let \(0<\delta\le1/2\), \(0<\theta<1\), and set
\[
 A=(P-\delta I)(E-\theta I),\qquad
 a=\sqrt{\delta(1-\delta)\theta(1-\theta)}.
\]
For every integer \(p\ge1\), \(\tr A^{2p}\ge-2r a^{2p}\).
If \(0<\eta<1\) and \(\delta\le\theta\eta^2/64\), then pointwise
\begin{equation}\label{eq:spectral-transfer}
 1_{\{\norm{\theta^{-1}X^TEX-I_r}>\eta\}}
          (\theta\eta/2)^{2p}
                \le\tr A^{2p}+2r a^{2p}.
\end{equation}
\end{lemma}
\begin{proof}
Fix an integer \(p\ge1\). A common invariant-block
decomposition of \(P\) and \(E\) controls cancellation in
\(\tr A^{2p}\). The fixed eigenvalue product on each
two-dimensional block bounds its possible negative contribution.
Under the additional small-\(\delta\) hypothesis, a spectral
deviation forces one block to contribute more than
\((\theta\eta/2)^{2p}\). The blockwise lower bounds control the
remaining contributions and yield the stated inequality.

\paragraph{Compressed spectrum.}
The compressed operator is \(PEP\) restricted to
\(\operatorname{range}P\). Since \(X^TX=I_r\) and \(P=XX^T\),
the map \(v\mapsto Xv\) is an isometry onto that space, and
\[
 PEPX=X(X^TEX).
\]
Thus its eigenvalues \(\lambda_1,\ldots,\lambda_r\) are precisely
those of \(X^TEX\). They lie in \([0,1]\), because
\(0\le u^TEu\le\|u\|^2\). In particular,
\[
 \norm{\theta^{-1}X^TEX-I_r}
       =\theta^{-1}\max_{1\le i\le r}|\lambda_i-\theta|.
\]

\paragraph{Invariant blocks.}
The standard two-projection decomposition
\cite[Theorem~2]{halmos1969} is constructed directly in finite
dimensions here to include the endpoint eigenspaces.
Diagonalize \(PEP\) on \(\operatorname{range}P\).
For a unit eigenvector \(u\) with eigenvalue \(0<\lambda<1\), put
\[
 w=\frac{Eu-\lambda u}{\sqrt{\lambda(1-\lambda)}}.
\]
Here \(\lambda u=PEu\), so \(w\) is the normalized component of
\(Eu\) perpendicular to \(\operatorname{range}P\).
Indeed, \(\|Eu\|^2=\langle u,Eu\rangle=\lambda\), and hence
\[
 P(Eu-\lambda u)=0,\qquad
 \|Eu-\lambda u\|^2=\lambda-2\lambda^2+\lambda^2
                         =\lambda(1-\lambda).
\]
It follows that \(Pw=0\), \(\norm w=1\), and \(u,w\) are orthonormal.
The definition of \(w\) and \(E^2=E\) give
\[
 Eu=\lambda u+\sqrt{\lambda(1-\lambda)}\,w,\qquad
 Ew=\frac{(1-\lambda)Eu}{\sqrt{\lambda(1-\lambda)}}
    =\sqrt{\lambda(1-\lambda)}\,u+(1-\lambda)w.
\]
Thus their span is invariant under both projections, whose matrices
in this basis are
\begin{equation}\label{eq:two-projection-block}
 P=\begin{pmatrix}1&0\\0&0\end{pmatrix},\qquad
 E=\begin{pmatrix}\lambda&\sqrt{\lambda(1-\lambda)}\\
                  \sqrt{\lambda(1-\lambda)}&1-\lambda\end{pmatrix}.
\end{equation}
Distinct eigenbasis vectors \(u,u'\) with eigenvalues
\(\lambda,\lambda'\in(0,1)\) give orthogonal pairs, even when
\(\lambda=\lambda'\). To see this, use
\[
 \langle u,Eu'\rangle
   =\langle u,PEPu'\rangle
   =\lambda'\langle u,u'\rangle=0,\qquad
 \langle Eu,Eu'\rangle=\langle u,Eu'\rangle=0.
\]
Together with the definitions of \(w,w'\), these identities give
\(\langle u,w'\rangle=\langle w,u'\rangle=\langle w,w'\rangle=0\).

For the endpoint eigenvalues \(\lambda=0,1\), no second vector \(w\)
is needed. In either case \(Pu=u\). If \(\lambda=0\), then
\(\|Eu\|^2=\lambda=0\), so \(Eu=0\). If \(\lambda=1\), then
\(\|(I-E)u\|^2=1-\lambda=0\), so \(Eu=u\).
The corresponding eigenspaces are therefore
\(\operatorname{range}P\cap\ker E\) and
\(\operatorname{range}P\cap\operatorname{range}E\), respectively.
Each endpoint eigenvector in the chosen orthonormal eigenbasis spans a
one-dimensional subspace invariant under both \(P\) and \(E\).

Every vector of the eigenbasis of
\(\operatorname{range}P\) now lies in a one- or two-dimensional invariant block.
It remains to handle the orthogonal complement of all these blocks.
Because the blocks together contain \(\operatorname{range}P\),
this complement lies in \(\ker P\), so \(P\) is zero there.
To check that \(E\) preserves it, take a vector \(v\) in the complement
and a vector \(z\) in any constructed block. Then
\[
 \langle Ev,z\rangle=\langle v,Ez\rangle=0,
\]
where the first equality uses symmetry of \(E\), and the last
follows because \(Ez\) stays in the same block.
Thus the restriction of \(E\) to the complement is again an
orthogonal projection. An orthonormal eigenbasis of this restriction
splits the complement into one-dimensional blocks on which \(P=0\)
and \(E=0\) or \(E=1\). This completes the decomposition of the
entire ambient space.

Finally, a two-dimensional block is created only for an eigenbasis
vector whose eigenvalue lies in \((0,1)\). There are \(r\) vectors
in the eigenbasis of \(\operatorname{range}P\), so the number of
two-dimensional blocks is at most \(r\).

\paragraph{Block polynomial.}
On \eqref{eq:two-projection-block}, the restriction of \(A\) is
\[
 \begin{pmatrix}
 (1-\delta)(\lambda-\theta)&
       (1-\delta)\sqrt{\lambda(1-\lambda)}\\
 -\delta\sqrt{\lambda(1-\lambda)}&
       -\delta(1-\lambda-\theta)
 \end{pmatrix}.
\]
Its trace is
\((1-\delta)(\lambda-\theta)-\delta(1-\lambda-\theta)
=\lambda-\delta-\theta+2\delta\theta\).
Its determinant is \(a^2\): the two centered factors have
determinants \(-\delta(1-\delta)\) and \(-\theta(1-\theta)\).
Consequently its characteristic polynomial is
\begin{equation}\label{eq:block-polynomial}
 z^2-tz+a^2,\qquad t=\lambda-\delta-\theta+2\delta\theta.
\end{equation}
Write its roots as \(z_1,z_2\). The product and sum of the roots satisfy
\[
 z_1z_2=a^2,\qquad
 z_1+z_2=t=(\lambda-\theta)-\delta(1-2\theta).
\]
Thus the product is independent of \(\lambda\), while the sum records
the deviation \(\lambda-\theta\) up to the shift
\(-\delta(1-2\theta)\).

\paragraph{Trace lower bound.}
To prove \(\tr A^{2p}\ge-2ra^{2p}\), consider a two-dimensional
block. Its contribution to \(\tr A^{2p}\) equals
\(z_1^{2p}+z_2^{2p}\), including when the block has a defective
repeated root. Indeed, Cayley--Hamilton gives the recurrence
\(S_j=tS_{j-1}-a^2S_{j-2}\), \(j\ge2\), for both the power traces
and the root-power sums, with the same initial values
\(S_0=2\), \(S_1=t\).

If the roots are nonreal, they
are conjugates, and \(z_1z_2=|z_1|^2=a^2\) gives
\[
 z_1^{2p}+z_2^{2p}
       =2\operatorname{Re}(z_1^{2p})\ge-2a^{2p}.
\]
If the roots are real, their even-power sum is nonnegative.
The remaining one-dimensional eigenvalues are among
\[
 -(1-\delta)\theta,\quad (1-\delta)(1-\theta),\quad
                  -\delta(1-\theta),\quad \delta\theta,
\]
corresponding to the four possibilities
\((P,E)=(1,0),(1,1),(0,1),(0,0)\) on a common eigenvector.
Their even powers are nonnegative. Summing over at most \(r\)
two-dimensional blocks proves \(\tr A^{2p}\ge-2ra^{2p}\).

\paragraph{Spectral deviation.}
For the second conclusion, assume \(0<\eta<1\) and
\(\delta\le\theta\eta^2/64\). Suppose the bad event in
\eqref{eq:spectral-transfer} occurs.
Since the deviation norm equals
$\theta^{-1}\max_i|\lambda_i-\theta|$, there is a compressed eigenvalue with
\(|\lambda-\theta|>\theta\eta\). Set \(b=\theta\eta\).
The following cases identify a block contributing more than \((b/2)^{2p}\).

For \(0<\lambda<1\), the corresponding block has polynomial
\eqref{eq:block-polynomial}. The hypothesis and \(0<\eta<1\) give
\(\delta\le b\eta/64<b/64\). Since \(|1-2\theta|\le1\),
\[
 \begin{split}
 |t|&=|(\lambda-\theta)-\delta(1-2\theta)|
       \ge|\lambda-\theta|-\delta|1-2\theta|
       >b-\delta>63b/64,\\
 a^2&\le\delta\theta\le\theta^2\eta^2/64=b^2/64.
 \end{split}
\]
In particular,
\[
 |t|>63b/64,\qquad a\le\sqrt{\delta\theta}\le b/8.
\]
It follows that \(|t|>2a\), so the discriminant \(t^2-4a^2\)
is positive and both roots are real. They have the same sign
because \(z_1z_2=a^2>0\). Also,
\[
 \min\{|z_1|,|z_2|\}\le\sqrt{|z_1z_2|}=a,\qquad
 |z_1|+|z_2|=|t|.
\]
Thus the larger absolute value is at least \(|t|-a\).
Since \(|t|-a>55b/64>b/2\) and both even powers are nonnegative, this
block contributes more than \((b/2)^{2p}\).

The endpoint cases give the same conclusion in one dimension.
For \(\lambda=0\), the eigenvalue of \(A\) is
\(-(1-\delta)\theta\), with magnitude
\((1-\delta)\theta\ge\theta/2>b/2\).
For \(\lambda=1\), the bad-event condition gives
\(1-\theta=|\lambda-\theta|>b\). The corresponding eigenvalue
\((1-\delta)(1-\theta)\) is therefore at least
\((1-\theta)/2>b/2\). These bounds use
\(\delta\le1/2\) and \(\eta<1\), and apply throughout \(0<\theta<1\).
These cases also cover \(E=0\) and \(E=I\).

\paragraph{Final comparison.}
On the bad event, one block contributes more than \((b/2)^{2p}\).
After excluding this block, at most \(r\) two-dimensional blocks
remain, each contributing at least \(-2a^{2p}\); every remaining
one-dimensional block contributes a nonnegative amount.
Thus all other blocks together contribute at least \(-2ra^{2p}\),
so
\[
 \tr A^{2p}+2ra^{2p}>(b/2)^{2p}=(\theta\eta/2)^{2p}.
\]
On the complementary event the indicator is zero, and the lower
trace bound makes the right-hand side nonnegative.
This proves \eqref{eq:spectral-transfer}.
\end{proof}

The event bound in Lemma~\ref{lem:transfer} applies to both spectral
edges because it depends on $|\lambda-\theta|$. Its right-hand side
is nonnegative after the explicit compensation, even on the event
where the Gram matrix is close to $I_r$.

\begin{remark}[Centered projection traces]
Kunisky \cite[Theorem~1.12 and Sections~4--5]{kunisky_manova}
also passes from centered alternating traces to a spectral edge.
His theorem assumes fixed positive limiting rank proportions, one
equal to $1/2$, and convergence of the empirical spectral distribution
to the MANOVA law. Bounds at moment orders growing faster than the
logarithm of the dimension then give convergence in probability of the largest
eigenvalue. The transfer uses recursions for moments of the positive
compressed operator. In Lemma~\ref{lem:transfer}, the fixed product
of the two roots instead gives a pointwise compensation for negative
trace contributions and controls both edges under
$\delta\le\theta\eta^2/64$.
\end{remark}

Now return to the Walsh projection $P=XX^T$, with
$X=HD_2HD_1V$ and $\delta=r/n$. With the independent Bernoulli
selector $E$ from the start of this section, put
\[
 \kappa=n\theta,\qquad K_2=4^{24},\qquad K_0=576K_2.
\]
\begin{proposition}\label{prop:trace}
For an integer \(p\ge2\), suppose
\begin{equation}\label{eq:trace-hypotheses}
 0<\theta<1,\qquad \kappa\ge r,\qquad
                         r\ge2(4p+1)^{1000}.
\end{equation}
Then
\begin{equation}\label{eq:trace-bound}
                  \left|\E\tr A^{2p}\right|
                           \le K_0^p(\delta\theta)^p.
\end{equation}
\end{proposition}

Notice that the absolute value in \eqref{eq:trace-bound} is outside the expectation.

The proof of Proposition~\ref{prop:trace} is given in
Section~\ref{sec:trace-proof}, using the cumulant estimates in
Section~\ref{sec:cumulants}. Combining Proposition~\ref{prop:trace}
with Lemma~\ref{lem:transfer} gives the following Bernoulli
embedding estimate.

\begin{samepage}
\begin{lemma}\label{cor:bernoulli}
Let
\[
 p=\lceil\log_2(3000r)\rceil,
 \qquad
 r\ge2(4p+1)^{1000}.
\]
For \(0<\eta<1\), \(0<\theta<1\), and
\(\kappa=n\theta\ge64K_0r/\eta^2\), the exact random projection \(P\)
and independent Bernoulli coordinate projection \(E\) satisfy
\begin{equation}\label{eq:bernoulli-bound}
 \Prob\{\norm{\theta^{-1}X^TEX-I_r}>\eta\}\le1/1000.
\end{equation}
\end{lemma}
\end{samepage}
\begin{proof}
Let
$$
 \mathcal B
   =\left\{
      \norm{\theta^{-1}X^TEX-I_r}>\eta
     \right\}.
$$
Since \(\kappa=n\theta\) and \(\delta=r/n\), the assumption
\(\kappa\ge64K_0r/\eta^2\) implies
$$
 \kappa\ge r,
 \qquad
 \delta=\frac{r\theta}{\kappa}
     \le \frac{\theta\eta^2}{64K_0}
     \le \frac{\theta\eta^2}{64}.
$$
In particular, $0<\delta<1/64<1/2$. Hence Proposition~\ref{prop:trace} and
Lemma~\ref{lem:transfer} both apply. 
Taking expectations in the pointwise inequality
\eqref{eq:spectral-transfer} gives
$$
 \Prob(\mathcal B)\left(\frac{\theta\eta}{2}\right)^{2p}
 \le
 \E\tr A^{2p}+2r a^{2p}.
$$
Proposition~\ref{prop:trace},
\(a^2\le\delta\theta\), and
\(\E\tr A^{2p}\le|\E\tr A^{2p}|\) together give
$$
 \Prob(\mathcal B)\left(\frac{\theta\eta}{2}\right)^{2p}
 \le
 (K_0^p+2r)(\delta\theta)^p.
$$
Since
$
 \delta\theta=\frac{r\theta^2}{\kappa}
$, 
division by \((\theta\eta/2)^{2p}\) yields
$$
 \Prob(\mathcal B)
 \le
 (K_0^p+2r)
 \left(\frac{4r}{\kappa\eta^2}\right)^p.
$$
The width assumption gives
$
 \frac{4r}{\kappa\eta^2}\le\frac1{16K_0}
$, 
and therefore, since \(K_0\ge1\),
$$
 \Prob(\mathcal B)
 \le
 (1+2r)16^{-p}
 \le
 3r\,2^{-p}.
$$
Finally,
\(p=\lceil\log_2(3000r)\rceil\) implies
\(2^{-p}\le(3000r)^{-1}\), so

$$
 \Prob(\mathcal B)\le 1/{1000}.
$$

\end{proof}

Lemma~\ref{cor:bernoulli} supplies the large-rank Bernoulli estimate.
The remaining sampling and rank cases will be treated in
Section~\ref{sec:main-proof}. The next section proves the entry estimates
needed for Proposition~\ref{prop:trace}.

\section{Joint cumulants of the transformed projection}
\label{sec:cumulants}

The proof of Proposition~\ref{prop:trace} expands the trace into
products of entries of the centered projection $R=P-\delta I_n$.
These entries are dependent. Joint cumulants organize their joint
moments into contributions from blocks of entries: the second
cumulant is covariance, and higher cumulants separate a collection's
dependence from products of smaller subcollections. The target estimate is
Lemma~\ref{lem:cumulants}, which bounds joint cumulants of the scaled
entries $nP_{ij}$ by $r$ times a factor depending only on their order
and gives a Walsh-symmetry condition for their vanishing. The same
estimates apply to $R$ at orders at least two; its first cumulants
vanish by centering.

The frame $V$ remains fixed throughout, as in Theorem~\ref{thm:main}.
Section~\ref{sec:graph-contractions} proves the deterministic
graph-sum bound that supplies the factor $r$.
Section~\ref{sec:cumulant-identities} develops the moment and product
cumulant identities. The moment identity will expand the trace
moments, while the product identity will turn the sign expansion of
$P$ into connected graph sums. Section~\ref{sec:projection-cumulants}
combines the graph bound with these identities, then uses Walsh
symmetry and computes the mean to obtain the estimates needed for $R$.

\subsection{Graph contractions}\label{sec:graph-contractions}

The sign expansion will leave deterministic sums of matrix-entry
products with shared indices. The graph representation of such a sum
is described first, followed by the bound needed for the graphs that arise.

\paragraph{A graph contraction.}
For positive integers $N_u,N_v,N_w$, let
$A\in\R^{N_v\times N_u}$, $B\in\R^{N_w\times N_v}$, and
$C\in\R^{N_w\times N_u}$ be fixed matrices. Consider
\[
 S=\sum_{i=1}^{N_u}\sum_{j=1}^{N_v}\sum_{k=1}^{N_w}
       C_{ki}B_{kj}A_{ji}.
\]
The shared indices are recorded by the following triangle:
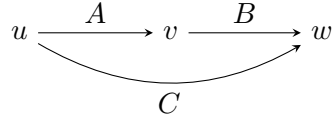
\begin{center}
\begin{tikzpicture}[>=stealth]
\node (u) at (0,0) {$u$};
\node (v) at (2,0) {$v$};
\node (w) at (4,0) {$w$};
\draw[->] (u) -- node[above] {$A$} (v);
\draw[->] (v) -- node[above] {$B$} (w);
\draw[->] (u) to[bend right=30] node[below] {$C$} (w);
\end{tikzpicture}
\end{center}
A \emph{labeling} $(i_u,i_v,i_w)=(i,j,k)$ selects the summand
$C_{ki}B_{kj}A_{ji}$. Assigning $A$ to $u\to v$ means that this
edge contributes $A_{ji}$: its column index is at $u$ and its row
index is at $v$. The edge matrices and coordinate ranges are fixed;
the labeling varies. The \emph{graph contraction} sums the products
over all labelings. Thus the graph records shared indices among
matrix-entry factors; the edge matrices are not adjacency matrices.
Here contraction means index summation, not contracting an edge
of the graph.

The following definition fixes the general convention for later use.

\begin{definition}[Graph contraction]\label{def:graph-contraction}
The convention is the real-coordinate form of a \emph{graph of matrices} in
Mingo and Speicher \cite[Definition~8 and Eq.~(9), pp.~2279--2280]{mingo},
with optional scalar vertex weights included as diagonal factors.
Let $\mathcal V$ be the vertex set
and $\mathcal E$ the list of edge occurrences of a finite multigraph,
allowing parallel edges and loops. Each vertex $u$ has a positive integer
coordinate dimension $N_u$ and a range $\{1,\ldots,N_u\}$ from which
its label $i_u$ is chosen,
and each oriented edge $e=(u,v)$ carries
a fixed matrix $M_e:\R^{N_u}\to\R^{N_v}$. A \emph{vertex weight}
is an optional fixed scalar function $\omega_u$ contributing the
factor $\omega_u(i_u)$. The contraction is
\begin{equation}\label{eq:graph-contraction}
 \sum_{\substack{1\le i_u\le N_u\\u\in\mathcal V}}
       \prod_{e=(u,v)\in\mathcal E}(M_e)_{i_v,i_u}
       \prod_{u\in\mathcal V}\omega_u(i_u).
\end{equation}
Parallel edges give separate factors, while a loop at $u$ gives
$(M_e)_{i_u,i_u}$. A loop counts twice toward the degree, and
connectedness ignores orientation. The sum is unrestricted:
labels at distinct vertices may coincide when their coordinate
sets coincide. It is also signed: no absolute value is taken inside
the product or the sum.
\end{definition}

\paragraph{Boundary matrices.}
Return to the triangle and fix $i,k$, summing only $j$. This gives
\[
 T_{ki}=\sum_{j=1}^{N_v}C_{ki}B_{kj}A_{ji}
       =C_{ki}(BA)_{ki}.
\]
The vertices $u,w$ are the \emph{boundary vertices}, whose free
indices index the columns and rows of $T$; $v$ is an \emph{internal
vertex}, whose index is summed out. Thus the boundary matrix is
the entrywise product of $C$ and $BA$, whereas the full contraction
is $S=\one_{N_w}^TT\one_{N_u}$. Here $\one$ denotes an all-ones
vector of the indicated dimension. A weight $\omega_v$ at the
internal vertex inserts $\omega_v(j)$ into the sum, replacing $BA$
by $B\diag(\omega_v(1),\ldots,\omega_v(N_v))A$. This also explains
how vertex weights can be absorbed into diagonal edge matrices.

\begin{definition}[Boundary matrix]\label{def:boundary-matrix}
For a graph as in Definition~\ref{def:graph-contraction},
take all vertex weights to be one and choose
distinct vertices $s,t$. Fix $i_s=a$ and $i_t=b$ and sum all other
labels to define $T:\R^{N_s}\to\R^{N_t}$ by
\[
 T_{ba}=
 \sum_{(i_u)_{u\notin\{s,t\}}}
       \prod_{e=(u,v)\in\mathcal E}(M_e)_{i_v,i_u},
 \qquad i_s=a,\quad i_t=b,
\]
where every summed label ranges over its coordinate set. The vertices
$s$ and $t$ are called the input and output, respectively; summing their labels recovers
the full contraction as $\one_{N_t}^TT\one_{N_s}$.
For an input-output graph, this is the one-input, one-output
operator defined by Mingo and Speicher \cite[Eq.~(11), p.~2281]{mingo}.
Here the finite sum defines $T$ before any input-output hypothesis
is imposed; that hypothesis is needed for the operator-norm bound.
\end{definition}

\paragraph{Input-output form.}
\begin{definition}[Input-output graph]\label{def:input-output}
In the one-input, one-output case needed here, an \emph{input-output
graph} is a directed acyclic graph with distinct vertices $s,t$
such that $s$ has only outgoing edges, $t$ has only incoming edges,
and every vertex lies on a directed path from $s$ to $t$
\cite[Definition~10, p.~2280]{mingo}.
\end{definition}
The triangle above already has this
form with $s=u$ and $t=w$: its underlying undirected graph is a
cycle, but its displayed orientation has no directed cycle.

For an input-output graph in Definition~\ref{def:input-output},
Mingo and Speicher's operator bound
\cite[Theorem~11(1), p.~2281]{mingo} states that the boundary matrix satisfies
\[
 \|T\|_2\le\prod_{e\in\mathcal E}\|M_e\|_2.
\]
For the triangle, the right-hand side is
$\|A\|_2\|B\|_2\|C\|_2$; there is no factor $N_v$ from the
internal sum. The original graph need not already be in input-output
form. A \emph{bridge} is an edge whose deletion disconnects a
connected graph. Mingo and Speicher's graph conversion lemma
\cite[Lemma~16, p.~2286]{mingo} states that a connected graph without bridges can be modified
into input-output form with any two prescribed distinct vertices
as its sole input and output. The coordinate dimensions may differ,
and the edge matrices may be rectangular.

\paragraph{Vertex splitting.}
The allowed modifications are described in
\cite[Definition~13 and Proposition~14, p.~2285]{mingo}. In the triangle
example, split the internal vertex $v$ into $v_1,v_2$, attach $A$
to $v_1$ and $B$ to $v_2$, and join them by the identity $I_{N_v}$:
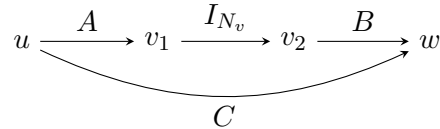
\begin{center}
\begin{tikzpicture}[>=stealth]
\node (u) at (0,0) {$u$};
\node (v1) at (1.8,0) {$v_1$};
\node (v2) at (3.6,0) {$v_2$};
\node (w) at (5.4,0) {$w$};
\draw[->] (u) -- node[above] {$A$} (v1);
\draw[->] (v1) -- node[above] {$I_{N_v}$} (v2);
\draw[->] (v2) -- node[above] {$B$} (w);
\draw[->] (u) to[bend right=25] node[below] {$C$} (w);
\end{tikzpicture}
\end{center}
The boundary entry is unchanged, because
\[
 \sum_{j_1,j_2=1}^{N_v}
 C_{ki}B_{k j_2}(I_{N_v})_{j_2j_1}A_{j_1i}
 =\sum_{j=1}^{N_v}C_{ki}B_{kj}A_{ji}=T_{ki}.
\]
The identity edge enforces equality of the two new labels. This
preserves each boundary entry, not only the sum $S$, and adds a
matrix of norm one. More generally, the prescribed modifications
distribute incident edges between copies of a vertex and connect
the copies by identity edges. One may also reverse an edge and
transpose its matrix, preserving its entry and operator norm.
With the designated boundary labels held fixed, these operations
preserve the boundary matrix, the full graph sum, and the product
of edge norms.

Lemma~\ref{lem:graph} will use the factorization $P_0=VV^T$ and
vertex splitting to make both boundary dimensions equal to $r$,
while the original indices range over $n$ coordinates.

\begin{lemma}\label{lem:graph}
Let a connected finite multigraph have positive even degrees. Assign
a matrix of operator norm at most one to each edge, with compatible
coordinate dimensions at its endpoints. Suppose one edge matrix is a
real rank-\(r\) orthogonal projection \(P_0=VV^T\), \(r\ge1\).
At each vertex also allow a coordinate weight of absolute value at most
one. The contraction of Definition~\ref{def:graph-contraction},
given by \eqref{eq:graph-contraction}, has absolute value
at most \(r\).
\end{lemma}
\begin{proof}
\begin{samepage}
The construction creates two boundary vertices of dimension $r$, without
changing the contraction. Mingo and Speicher's Lemma~16
\cite{mingo} will put the graph in input-output form with these
vertices as its boundaries. Their Theorem~11(1) will then bound
the associated matrix $T:\R^r\to\R^r$ by $\|T\|_2\le1$.
The remaining factor $r$ comes from summing the two boundary
labels, as verified at the end of the proof.
\par
\end{samepage}

A connected even-degree graph has no bridge. Indeed, if deleting one
edge separated a set of vertices from the rest, the degree sum on
that set would be twice its number of internal edges plus one.
This is odd, whereas a sum of even degrees is even.
Represent a vertex weight $\omega_v(i_v)$ by a loop carrying
$\diag(\omega_v)$. Its entry at that vertex is exactly
$\omega_v(i_v)$ and its norm is $\max_{i_v}|\omega_v(i_v)|\le1$.
Thus absorbing the weights preserves every summand of the
contraction. Each added loop contributes two to the degree, so
the graph remains connected with positive even degrees.

The projection factorization supplies a small coordinate space.
Choose the columns of $V$ to be an orthonormal basis of
$\operatorname{range}P_0$. Then $V^TV=I_r$, so
$\|Vx\|_2=\|x\|_2$ for $x\in\R^r$ and
$\|V\|_2=\|V^T\|_2=1$.
Replace the distinguished edge entry by
\[
 (P_0)_{ij}=\sum_{a=1}^r V_{ia}V_{ja}.
\]
This replaces the distinguished edge by a two-edge path through
a new degree-two vertex of dimension \(r\). Mingo and Speicher's
Lemma~16 requires two distinct vertices for the input and output,
not just this one.
Split it into two dimension-\(r\) vertices, put one incident edge
at each, and join them by \(I_r\). In coordinates, the replacement is
\[
 (P_0)_{ij}
   =\sum_{a,b=1}^r V_{ia}(I_r)_{ab}V_{jb}.
\]
If the distinguished projection acts on $\R^N$, the two rectangular
matrices are $V:\R^r\to\R^N$ and $V^T:\R^N\to\R^r$.
Call the new vertices $s,t$, with labels $a,b$, respectively.
Both labels range over exactly $r$ coordinates. Since
$(I_r)_{ab}$ is zero unless $a=b$, summing the new labels recovers
$(P_0)_{ij}$ for every fixed pair of old endpoint labels $i,j$.
All other factors are unchanged, so the full signed sum is unchanged.
The resulting graph is still connected with even degrees: the old
endpoint degrees are unchanged, and each new vertex has degree two.
If the original edge was a loop, the replacement is a three-edge
cycle through the old vertex, so the same degree check applies.
Thus the modified graph still has no bridge.

The graph is now connected and has no bridge, and $s,t$ are
distinct. These are precisely the hypotheses of Mingo and
Speicher's Lemma~16 \cite{mingo}, which gives a modification in
the input-output form of Definition~\ref{def:input-output},
with $s$ as its sole input and $t$ as its sole
output. The modifications reverse edges and transpose their
matrices, or split vertices using identity edges. By their
Proposition~14 \cite{mingo}, they preserve the full graph sum
and the product of edge norms. Transposition does not change a
norm, and every inserted identity has norm one. Splitting copies
the coordinate space of a vertex, so both boundary dimensions remain $r$.

In this input-output graph, use Definition~\ref{def:boundary-matrix}
to define \(T:\R^r\to\R^r\) by taking $T_{ba}$
to be the sum over all other labels with $a,b$ fixed at $s,t$. Mingo and
Speicher's Theorem~11(1) \cite{mingo} applies to these compatible
coordinate spaces and edge matrices. It gives
$\|T\|_2\le\prod_e\|M_e\|_2\le1$: the original edge matrices
and weight loops have norm at most one, and the factors $V,V^T$
and inserted identities have norm one.
Summing $a,b$ therefore gives the original full signed contraction
\(\one_r^T T\one_r\), where $\one_r$ is the all-ones vector.
Cauchy--Schwarz and the definition of the operator norm give
$|\one_r^TT\one_r|\le\|\one_r\|_2^2\|T\|_2\le r$.
Here $\|\one_r\|_2^2=r$; no ambient-dimension factor is introduced.
\end{proof}

\subsection{Cumulant identities}\label{sec:cumulant-identities}

With the graph bound in hand, the next step is to explain why connected graphs
arise in a joint cumulant of projection entries. Each entry expands
into products of signs. The Leonov--Shiryaev formula for cumulants
of products retains precisely the factor partitions that connect
all the prescribed products
\cite[Eq.~(IV.d), p.~323]{leonov_shiryaev1959}.
After the partition notation is introduced, this formula,
the moment-cumulant identity, and the needed vanishing rules are collected in
Lemma~\ref{lem:cumulant-identities}. All references to
Leonov and Shiryaev below use the English translation.

\begin{definition}[Joint cumulant]\label{def:joint-cumulant}
Let $q\ge1$ and let $Y_1,\ldots,Y_q$ be scalar random variables
for which all joint moments appearing below are finite.
Their joint cumulant is defined by
the classical moment-partition formula
\cite[Eq.~(II.c), p.~321]{leonov_shiryaev1959}
\begin{equation}\label{eq:cumulant-definition}
 \cum(Y_1,\ldots,Y_q)=
 \sum_{\pi\in\mathcal P([q])}(-1)^{|\pi|-1}(|\pi|-1)!
                       \prod_{B\in\pi}\E\prod_{j\in B}Y_j.
\end{equation}
Here $[q]=\{1,\ldots,q\}$, and $\mathcal P([q])$ denotes the set
of partitions of $[q]$ into nonempty disjoint blocks. For a partition
$\pi$, $|\pi|$ is its number of blocks, and $B\in\pi$ denotes one
such block. Each block contributes the joint moment
$\E(\prod_{j\in B}Y_j)$; the outer product multiplies these moments
over all blocks of $\pi$. The notation $\cum(Y_j:j\in B)$ denotes
the joint cumulant of the variables whose indices belong to $B$.
\end{definition}

For example, the partitions of $[2]$ give
\[
 \cum(Y_1,Y_2)=\E(Y_1Y_2)-(\E Y_1)(\E Y_2).
\]
Thus the second cumulant is covariance. Higher joint cumulants
separate the dependence of a collection from products of its
smaller subcollections.

A joint cumulant of products can be expressed as a sum of products
of joint cumulants of the individual factors. The indices in a
partition refer to occurrences in an expression, even when the
same random variable occurs more than once.
Let $\mathcal I$ be a finite set of occurrence positions, and let
$(Z_a)_{a\in\mathcal I}$ be scalar random variables with the joint
moments needed in the formulas below. Fix a partition
$\tau=\{I_1,\ldots,I_q\}$ of $\mathcal I$, and put
$Y_f=\prod_{a\in I_f}Z_a$ for $1\le f\le q$.
Thus $I_f$ records the positions of the factors in the $f$th product.
The variables $Z_a$ need not be distinct: different positions may
carry the same random variable.

\begin{definition}[Join of partitions]\label{def:partition-join}
For partitions $\sigma,\tau$ of a finite nonempty set $\mathcal I$,
the join of $\sigma$ and $\tau$,
$\sigma\vee\tau$, is the partition generated by both identifications:
two positions belong to the same join block if they can be
connected by a chain of $\sigma$-blocks or $\tau$-blocks.
Write $\one_{\mathcal I}=\{\mathcal I\}$ for the one-block partition.
For the prescribed product groups $\tau$, the condition
$\sigma\vee\tau=\one_{\mathcal I}$ expresses the
\emph{indecomposability} of the factor partition in
Leonov and Shiryaev's terminology
\cite[Remark~5, p.~323]{leonov_shiryaev1959}.
\end{definition}

In the product formula below, $\tau$ stays fixed, whereas $\sigma$
ranges over partitions of the factor positions to specify the
cumulants on the right-hand side.

\begin{lemma}[Moment and product cumulant identities]
\label{lem:cumulant-identities}
Let $q\ge1$. All variables below are scalar random variables for
which the joint moments in the stated formulas are finite.
Cumulants and joins are as in Definitions~\ref{def:joint-cumulant}
and~\ref{def:partition-join}.
\begin{enumerate}
\item For $Y_1,\ldots,Y_q$, the inverse, or moment-cumulant,
identity \cite[Eq.~(I.c), p.~321]{leonov_shiryaev1959} is
\begin{equation}\label{eq:moment-cumulant}
 \E\prod_{j=1}^qY_j
    =\sum_{\pi\in\mathcal P([q])}
       \prod_{B\in\pi}\cum(Y_j:j\in B).
\end{equation}
\item For a partition $\tau=\{I_1,\ldots,I_q\}$ of the factor
positions $\mathcal I$ and variables $(Z_a)_{a\in\mathcal I}$,
the Leonov--Shiryaev formula
\cite[Eq.~(IV.d), p.~323]{leonov_shiryaev1959}
reads
\begin{equation}\label{eq:product-cumulant}
 \cum\left(\prod_{a\in I_1}Z_a,\ldots,
                 \prod_{a\in I_q}Z_a\right)
 =\sum_{\substack{\sigma\in\mathcal P(\mathcal I)\\
                    \sigma\vee\tau=\one_{\mathcal I}}}
       \prod_{B\in\sigma}\cum(Z_a:a\in B).
\end{equation}
\item Joint cumulants are multilinear. A joint cumulant vanishes
when its arguments split into two nonempty independent groups.
At order $q\ge2$, adding deterministic constants to the arguments
does not change the cumulant. An odd-order cumulant of a random
variable symmetric about zero is zero.
\end{enumerate}
\end{lemma}

The join condition says that the blocks of $\sigma$, together with
the prescribed product groups, connect all $q$ products. For
instance, if a sign occurs twice in one product and once in another,
it still occupies three positions in $\mathcal I$; the formula
sums partitions of these three positions, not partitions of the
single sign name.
For a concrete example of the join condition, take
$I_1=\{1,2\}$ and $I_2=\{3,4\}$.
Then $\sigma=\{\{1,3\},\{2\},\{4\}\}$ joins the two product
groups through the block $\{1,3\}$ and contributes the term
\[
 \cum(Z_1,Z_3)\,\E Z_2\,\E Z_4,
\]
since a one-variable cumulant is its expectation. In contrast,
$\sigma=\{\{1,2\},\{3,4\}\}$ leaves the groups separate and is
excluded. The join condition determines which terms occur; their
cumulant factors may impose further restrictions or make them zero.

\begin{proof}
These formulas are used only as finite polynomial identities in the
joint moments. First, consider the moment polynomial
\[
 M(t_1,\ldots,t_q)
    =\E\prod_{j=1}^q(1+t_jY_j)
    =\sum_{B\subseteq[q]}\left(\E\prod_{j\in B}Y_j\right)
                         \prod_{j\in B}t_j,
\]
where $t_1,\ldots,t_q$ are formal variables and the empty product
is one. Since $M$ has constant term one, its formal logarithm is
\[
 \log M=\sum_{j\ge1}\frac{(-1)^{j-1}}{j}(M-1)^j.
\]
Every monomial of $M-1$ has positive degree. Thus terms with
$j>q$ cannot contribute to the coefficient of $t_1\cdots t_q$.
For $1\le j\le q$, expanding $(M-1)^j$ amounts to choosing
an ordered list of $j$ nonempty subsets of $[q]$, one from each
factor. To obtain $t_1\cdots t_q$, each index must occur exactly
once in this list: the chosen subsets must be disjoint and their
union must be $[q]$. They therefore form an ordered partition.
Each unordered partition into $j$ blocks has $j!$ such orderings,
all contributing the same product of moments. If
$[t_1\cdots t_q]F$ denotes the coefficient of $t_1\cdots t_q$
in a formal series $F$, this gives
\[
 [t_1\cdots t_q](M-1)^j
   =j!\sum_{\substack{\pi\in\mathcal P([q])\\|\pi|=j}}
          \prod_{B\in\pi}\E\prod_{i\in B}Y_i.
\]
The logarithm supplies the factor $(-1)^{j-1}/j$, so each
partition receives the coefficient $(-1)^{j-1}(j-1)!$.
Summing over $j$ shows that the coefficient of $t_1\cdots t_q$
in $\log M$ is $\cum(Y_1,\ldots,Y_q)$ by
\eqref{eq:cumulant-definition}.
The same coefficient argument for any nonempty subset $B$ gives
$\cum(Y_j:j\in B)$ as the coefficient of $\prod_{j\in B}t_j$
in $\log M$. Now compare the coefficient of $t_1\cdots t_q$
in $M=\exp(\log M)$. Only products of monomials on disjoint
nonempty subsets can contribute. For a partition $\pi$, its
$|\pi|!$ block orderings cancel the factorial denominator in
the exponential series. The resulting coefficient is the
right-hand side of \eqref{eq:moment-cumulant}, proving part~1.

To verify the coefficient in
\eqref{eq:product-cumulant}, apply \eqref{eq:cumulant-definition}
to the products $Y_1,\ldots,Y_q$, and expand each resulting joint
moment using \eqref{eq:moment-cumulant} for the factors $Z_a$.
For an outer block $C\subseteq[q]$, write $U_C$ for the union
$\bigcup_{f\in C}I_f$ of its factor positions. Since the $I_f$
are disjoint, $\prod_{f\in C}Y_f=\prod_{a\in U_C}Z_a$.
The two substitutions give
\[
\begin{aligned}
 \cum(Y_1,\ldots,Y_q)
 &=\sum_{\pi\in\mathcal P([q])}(-1)^{|\pi|-1}(|\pi|-1)!
       \prod_{C\in\pi}\E\prod_{a\in U_C}Z_a\\
 &=\sum_{\pi\in\mathcal P([q])}(-1)^{|\pi|-1}(|\pi|-1)!
       \prod_{C\in\pi}\left(
          \sum_{\sigma_C\in\mathcal P(U_C)}
             \prod_{B\in\sigma_C}\cum(Z_a:a\in B)\right).
\end{aligned}
\]
For fixed $\pi$, the sets $U_C$, $C\in\pi$, partition
$\mathcal I$. Choosing one inner partition $\sigma_C$ for each
$C$ therefore gives a partition $\sigma$ of $\mathcal I$ by
taking all their blocks together. Conversely, a fixed $\sigma$
arises this way exactly when every one of its blocks lies inside
one $U_C$. In that case the inner partitions are uniquely
determined, so the product of inner sums contributes it once.
Keeping occurrence labels distinct, fix a partition $\sigma$ and
collect the terms with cumulant product
$\prod_{B\in\sigma}\cum(Z_a:a\in B)$.
The admissible outer partitions are therefore precisely those that
do not split a group of products already connected by $\sigma$
and $\tau$.

Put $b=|\sigma\vee\tau|$. Each of these $b$ connected groups
contains whole product groups $I_f$, since the join includes all
identifications from $\tau$. The admissible outer partitions are
therefore exactly the ways to partition these $b$ groups further.
There are $S(b,j)$ ways to obtain $j$ outer blocks, and each has
coefficient $(-1)^{j-1}(j-1)!$ in
\eqref{eq:cumulant-definition}. Hence the coefficient of the fixed
cumulant product is
\[
 \sum_{j=1}^b S(b,j)(-1)^{j-1}(j-1)!=\boldsymbol{1}_{\{b=1\}},
\]
where \(S(b,j)\) is the number of partitions of a $b$-element set
into \(j\) nonempty blocks. To see the cancellation explicitly, use
\[
 \sum_{b\ge j}S(b,j)\frac{z^b}{b!}
       =\frac{(e^z-1)^j}{j!}.
\]
Indeed, fix $j\ge1$ and expand each factor as
$e^z-1=\sum_{k\ge1}z^k/k!$. Multiplication and collection by
total degree give
\[
 \begin{split}
 (e^z-1)^j
 &=\sum_{k_1,\ldots,k_j\ge1}
       \frac{z^{k_1+\cdots+k_j}}{k_1!\cdots k_j!}\\
 &=\sum_{b\ge j}\left(
       \sum_{\substack{k_1+\cdots+k_j=b\\k_1,\ldots,k_j\ge1}}
       \frac{b!}{k_1!\cdots k_j!}\right)\frac{z^b}{b!}.
 \end{split}
\]
For fixed positive block sizes $k_1,\ldots,k_j$ summing to $b$,
the factor $b!/(k_1!\cdots k_j!)$ counts ways to assign $b$
distinct labels to $j$ ordered blocks of those sizes: choose the
labels of the first block, then the second, and so on.
Summing over all positive size lists therefore counts all
partitions of these labels into $j$ ordered nonempty blocks.
Every partition counted by $S(b,j)$ has exactly $j!$ block
orderings, even if some block sizes coincide, since its blocks
are distinct subsets. Hence
\[
 \sum_{\substack{k_1+\cdots+k_j=b\\k_1,\ldots,k_j\ge1}}
       \frac{b!}{k_1!\cdots k_j!}=j!S(b,j).
\]
Thus the coefficient of $z^b/b!$ in $(e^z-1)^j$ is $j!S(b,j)$;
division by $j!$ proves the stated generating-function identity.
Consequently,
\[
 \begin{split}
 \sum_{b\ge1}\left[\sum_{j=1}^b
   S(b,j)(-1)^{j-1}(j-1)!\right]\frac{z^b}{b!}
 &=\sum_{j\ge1}\frac{(-1)^{j-1}}{j}(e^z-1)^j\\
 &=\log\bigl(1+(e^z-1)\bigr)=z.
 \end{split}
\]
This is the coefficient comparison in \(\log(\exp z)=z\):
the coefficient is one when $b=1$ and zero when $b>1$.
Thus precisely the partitions with a connected join survive,
each with coefficient one, proving part~2. The comparison is
formal; each fixed coefficient involves only finitely many terms.

For part~3, multilinearity follows directly from
\eqref{eq:cumulant-definition}, since each argument occurs in
exactly one moment factor in every summand.
Suppose the arguments $Y_1,\ldots,Y_q$ split into two
independent families indexed by nonempty disjoint sets $A_1,A_2$
with $A_1\cup A_2=[q]$. The usual logarithm argument for the
vanishing of mixed cumulants
\cite[\S3.B, lemma on p.~324]{leonov_shiryaev1959}
uses the same moment polynomial $M$. Independence gives
\[
 M=M_{A_1}M_{A_2},\qquad
 M_{A_i}=\E\prod_{j\in A_i}(1+t_jY_j),\quad i=1,2,
 \qquad \log M=\log M_{A_1}+\log M_{A_2}.
\]
Each logarithm on the right involves only the variables from its
own family, so neither contains $t_1\cdots t_q$. Its coefficient,
and hence the mixed joint cumulant, is zero. Only terms through
total degree $q$ are needed in these formal logarithms; no
convergence assumption on a moment-generating function is used.

For $q\ge2$, a deterministic argument forms an independent group
from all the other arguments, so the corresponding cumulant is
zero. Multilinearity then proves invariance under adding constants.
Finally, if $Z$ and $-Z$ have the same law, the order-$q$ cumulant
of $Z$ equals $(-1)^q$ times itself by multilinearity. It is
therefore zero for odd $q$.
\end{proof}

\subsection{Cumulants of the transformed projection}
\label{sec:projection-cumulants}

Recall that $K=(\F_2)^m$, $n=2^m$, and
$H_{ab}=n^{-1/2}(-1)^{\langle a,b\rangle}$ is the normalized Walsh
matrix. The deterministic frame $V\in\R^{n\times r}$ satisfies
$V^TV=I_r$, and $\delta=r/n$. All diagonal entries of $D_1,D_2$
are mutually independent uniform Rademacher signs. In this subsection,
expectations and cumulants are taken over these two sign diagonals,
with $V$ and the matrix-entry indices fixed.

The two ingredients can now be combined. Expanding the entries of $P$
gives products of signs with deterministic matrix coefficients.
The product formula in Lemma~\ref{lem:cumulant-identities}
groups these coefficients into
connected graph sums, each bounded by $r$ through Lemma~\ref{lem:graph}.
Counting the surviving partitions gives the order-dependent factor
in the following estimate. Walsh symmetry then restricts
the contributing index combinations, and the mean calculation permits
the passage from $P$ to $R$.

\begin{samepage}
\begin{lemma}\label{lem:cumulants}
For the projection \(P=HD_2HD_1VV^TD_1HD_2H\), every integer
\(q\ge1\), and arbitrary fixed pairs of indices
\((i_f,j_f)\in K\times K\), $1\le f\le q$, with repetitions allowed,
\begin{equation}\label{eq:cumulant-bound}
 \left|\cum(nP_{i_1j_1},\ldots,nP_{i_qj_q})\right|
                      \le (2q)^{12q}r.
\end{equation}
The cumulant is zero unless
\begin{equation}\label{eq:xor-constraint}
                 \sum_{f=1}^q(i_f+j_f)=0\quad\hbox{in }K.
\end{equation}
Moreover, \(\E P=\delta I_n\).
\end{lemma}
\end{samepage}
The sum in \eqref{eq:xor-constraint} is coordinatewise addition
modulo two, or XOR. This is a necessary condition for a nonzero
cumulant; it does not assert nonvanishing when the sum is zero.
\begin{proof}
The proof first establishes the bound by expanding into connected graph sums,
then uses Walsh symmetry for the vanishing condition, and finally
computes the mean.
The combination of cumulant expansions and graph-sum bounds also
occurs in the trace-cumulant analysis of V\'azquez-Becerra
\cite[Section~4]{vazquez_becerra}; the calculation below keeps the
entry indices fixed and the dependence on $q$ and $r$ explicit.
Write \(P_0=VV^T\) and \(\chi_i(a)=(-1)^{\langle i,a\rangle}\).
For $\ell=1,2$ and $a\in K$, denote the diagonal signs by
$\xi_{\ell,a}=(D_\ell)_{aa}$. The exact entry expansion is
\begin{equation}\label{eq:entry-expansion}
 nP_{ij}=\sum_{a,b,x,y\in K}
 \chi_i(a)\chi_j(b)\,
 \xi_{2,a}\xi_{2,b}\xi_{1,x}\xi_{1,y}
                  H_{ax}(P_0)_{xy}H_{yb}.
\end{equation}
The factor \(n\) cancels the two outer Hadamard normalizations.
Both inner Hadamard matrices in \eqref{eq:entry-expansion} remain normalized.

For fixed summation indices in $q$ copies of
\eqref{eq:entry-expansion}, let $\mathcal I$ be the set of the
$4q$ sign occurrence positions. Set $\tau=\{I_1,\ldots,I_q\}$,
where $I_f$ contains the occurrences of
$\xi_{2,a_f},\xi_{2,b_f},\xi_{1,x_f},\xi_{1,y_f}$ in the
$f$th product. These are four positions even if some labels agree.
By multilinearity, the joint cumulant can be expanded into a sum over
the $q$ copies of the four indices, with their deterministic
coefficients taken outside the cumulant. Apply
Lemma~\ref{lem:cumulant-identities}, in the form
\eqref{eq:product-cumulant}, and then sum the deterministic
matrix and character coefficients. A nonzero partition
has a partition \(\alpha\) of its \(2q\) middle-sign ($D_2$) occurrences
and a partition \(\beta\) of its \(2q\) input-sign ($D_1$) occurrences,
all blocks even. Equal-coordinate occurrences within a block are identified.
Indeed, the vanishing rules in Lemma~\ref{lem:cumulant-identities}
force each block to involve just one sign
coordinate in one layer, and symmetry of that sign makes every
odd-order cumulant zero.
Different blocks may receive equal labels; no inequality between
their labels is imposed. For example, two different blocks can
both use $\xi_{2,a}$; their factors are still two cumulants in
\eqref{eq:product-cumulant}. Thus block membership imposes equalities
within a block but does not specify the full equality pattern of
all labels.

To display the resulting sum, write
$\kappa_b=\cum(\xi,\ldots,\xi)$ with $b$ arguments, where $\xi$
is a single uniform Rademacher sign. For fixed $\alpha,\beta$, let
\[
 S_{\alpha,\beta}
 =\sum_{\substack{a_f,b_f,x_f,y_f\in K\\1\le f\le q}}^{\prime}
   \prod_{f=1}^q
   \chi_{i_f}(a_f)\chi_{j_f}(b_f)
   H_{a_fx_f}(P_0)_{x_fy_f}H_{y_fb_f}.
\]
The prime imposes equality of coordinate labels within each block
of $\alpha$ and of $\beta$, and imposes no inequalities between
different blocks. Every block cumulant is then $\kappa_b$, where
$b$ is its size, independently of its coordinate label. Grouping
the finite sums by their partitions therefore gives
\[
 \begin{aligned}
 \cum(nP_{i_1j_1},\ldots,nP_{i_qj_q})
 &=\sum_{\substack{\alpha,\beta\ \text{with even block sizes}\\
             (\alpha\cup\beta)\vee\tau=\one_{\mathcal I}}}
   \left(\prod_{A\in\alpha}\kappa_{|A|}\right)
   \left(\prod_{B\in\beta}\kappa_{|B|}\right)
   S_{\alpha,\beta}.
 \end{aligned}
\]
Here $\alpha\cup\beta$ is the partition of all $4q$ occurrence
positions obtained by taking the blocks in both layers. It remains to
bound the deterministic sum $S_{\alpha,\beta}$ for each partition
pair in this display.

Before the block identifications, the $f$th product of matrix
entries describes the path
\[
 a_f\ \overset{H}{\longleftrightarrow}\ x_f
 \ \overset{P_0}{\longleftrightarrow}\ y_f
 \ \overset{H}{\longleftrightarrow}\ b_f.
\]
The connected-join condition in
Definition~\ref{def:partition-join} connects all $q$ such paths.
Indeed, the groups $I_f$ connect the occurrences
along each such path, and the blocks of $\alpha,\beta$ identify
vertices between paths. Having one join block is exactly what
makes their union connected after these identifications.

Each middle-sign position $a_f,b_f$ is incident to one edge, whereas
each input-sign position $x_f,y_f$ is incident to two. Merging the
positions in a block adds their degrees; an edge whose endpoints
merge becomes a loop and still contributes twice. Thus the resulting
graph has one vertex per middle-sign block, of degree equal to its
even size, and one per input-sign block, of degree twice its size.
It is connected with positive even degrees. Its matrices are
normalized \(H\)'s and \(P_0\)'s, all of norm one, and it contains
a rank-$r$ projection edge $P_0$ from each original path.
At each middle-sign vertex the outer characters multiply to a
coordinate weight of modulus one; input-sign vertices have weight one.
After the identifications, every block vertex has one freely summed
label in $K$. Thus $S_{\alpha,\beta}$ is precisely an unrestricted
signed graph contraction in Definition~\ref{def:graph-contraction},
and Lemma~\ref{lem:graph} gives $|S_{\alpha,\beta}|\le r$.
Absolute values have not been inserted into the matrix-entry products.
The graph bound applies to the full signed sum; it does not bound
the sum of the absolute values of its individual summands.

The order-\(b\) cumulant of a single Rademacher sign has absolute value
at most \(b^{2b}\): Definition~\ref{def:joint-cumulant}, in the
form \eqref{eq:cumulant-definition}, has at most \(b^b\) partitions, each
factorial coefficient is at most \(b!\le b^b\), and each moment has
absolute value at most one. For one layer, this gives
\[
 \prod_{A\in\alpha}|\kappa_{|A|}|
 \le\prod_{A\in\alpha}|A|^{2|A|}
 \le(2q)^{2\sum_{A\in\alpha}|A|}
 =(2q)^{4q},
\]
since its block sizes sum to $2q$. The same bound holds for $\beta$.
A set of $2q$ positions has at most $(2q)^{2q}$ partitions: after
numbering the positions, assigning each position the least number
in its block encodes the partition
by a function from that set to itself. Hence there are at most
$(2q)^{4q}$ partition pairs, even before the even-size and connected-join
restrictions. Applying the triangle inequality to the displayed
partition sum gives
\[
 \begin{aligned}
 \left|\cum(nP_{i_1j_1},\ldots,nP_{i_qj_q})\right|
 &\le \underbrace{(2q)^{4q}}_{\text{partition pairs}}
     \quad\underbrace{(2q)^{4q}(2q)^{4q}}_{\text{two coefficient products}}
     \quad\underbrace{r}_{\text{graph sum}}\\
 &=(2q)^{12q}r,
 \end{aligned}
\]
which proves \eqref{eq:cumulant-bound}.

The proof of \eqref{eq:xor-constraint} uses a symmetry that preserves the
law of $P$ but can reverse the sign of the target cumulant whenever
the XOR sum is nonzero. For $s\in K$, let
\(C_s=\diag(\chi_s(i):i\in K)\) and let $T_s$ translate coordinates:
$(T_su)_a=u_{a+s}$ for $u\in\R^n$. In particular $T_s^2=I_n$.
The character identity
$\chi_i(a+s)=\chi_i(a)\chi_i(s)$ gives the first two Walsh identities
below by comparing matrix entries:
\begin{equation}\label{eq:walsh-conjugation}
 C_sH=HT_s,\qquad T_sH=HC_s,\qquad
 C_sHD_2HD_1V=H(T_sD_2T_s)H(C_sD_1)V.
\end{equation}
For the last identity, insert $T_s^2=I_n$ next to $D_2$ and use
the first two identities:
\[
 \begin{aligned}
 C_sHD_2HD_1V
 &=HT_sD_2HD_1V\\
 &=H(T_sD_2T_s)T_sHD_1V\\
 &=H(T_sD_2T_s)H(C_sD_1)V.
 \end{aligned}
\]
The diagonal entries of $T_sD_2T_s$ are those of $D_2$ permuted
by $a\mapsto a+s$. The entries of $C_sD_1$ are those of $D_1$
multiplied by fixed signs. Both operations preserve the independent
uniform sign law within each diagonal, and the two new diagonals
remain independent because each depends on a different original
diagonal. Thus
\(C_sPC_s\) has the same law as \(P\) for the same fixed \(V\).
Entrywise,
\[
 (C_sPC_s)_{ij}=\chi_s(i)\chi_s(j)P_{ij}
              =\chi_s(i+j)P_{ij}.
\]
Write $c=\cum(nP_{i_1j_1},\ldots,nP_{i_qj_q})$. Equality in law
and multilinearity then give, for every $s\in K$,
\[
 c=\left(\prod_{f=1}^q\chi_s(i_f+j_f)\right)c
   =\chi_s\!\left(\sum_f(i_f+j_f)\right)c.
\]
If
$q_0=\sum_f(i_f+j_f)\ne0$, choose a nonzero coordinate of $q_0$
and let $s$ be the corresponding binary basis vector. Then
$\chi_s(q_0)=-1$, so the cumulant equals its negative and is zero.
This proves \eqref{eq:xor-constraint}.

Finally set \(B=HD_1P_0D_1H\), so that $P=HD_2BD_2H$.
Condition on $D_1$, which fixes $B$. The $(a,b)$ entry of $D_2BD_2$
is $B_{ab}\xi_{2,a}\xi_{2,b}$. Its conditional expectation is
$B_{aa}$ if $a=b$ and zero otherwise, by independence and centering
of the signs. Thus
\[
 \E_{D_2}[D_2BD_2\mid D_1]
      =\diag(B_{aa}:a\in K),
\]
where the right side is the diagonal matrix with the indicated
entries. To average these entries over $D_1$, expand
\[
 B_{aa}=\sum_{x,y\in K}
        H_{ax}H_{ya}(P_0)_{xy}\xi_{1,x}\xi_{1,y}.
\]
Only $x=y$ survives in expectation. Since $H_{ax}^2=1/n$ and
$\sum_x(P_0)_{xx}=\tr P_0=r$, this gives
\[
 \E_{D_1}B_{aa}
     =\sum_x H_{ax}^2(P_0)_{xx}=\frac rn.
\]
Taking the two expectations in turn and using $H^2=I_n$ now yields
\[
 \E P
 =H\,\diag(\E_{D_1}B_{aa}:a\in K)\,H
 =H\left(\frac rn I_n\right)H
 =\delta I_n.
\]
\end{proof}

Recall \(R=P-\delta I_n\). Every entry of \(R\), including every diagonal
entry, has mean zero. Cumulants of order at least two are unchanged
by this centering, by Lemma~\ref{lem:cumulant-identities}, so
\eqref{eq:cumulant-bound}--\eqref{eq:xor-constraint} hold for \(R\) at those orders.
The zero diagonal mean will remove singleton loop blocks in
the trace expansion.

\begin{remark}[Trace cumulants and entry cumulants]
The bounded trace-cumulant theorem of V\'azquez-Becerra
\cite[Lemma~6 and Theorem~7]{vazquez_becerra} uses simultaneous
conjugation invariance of the relative mixing matrices under every
signed permutation, together with bounds on their entry moments.
Its constants may depend on the fixed word lengths. The proof above
uses the two independent sign layers for an arbitrary fixed $V$,
retains a rank-$r$ edge in each graph sum, and gives the explicit
order dependence in \eqref{eq:cumulant-bound}. Its vanishing condition
comes from Walsh character symmetry. The discrete Fourier sums in
\cite[Section~3.2]{vazquez_becerra} are instead evaluated through
quadratic forms and Gauss sums; they do not supply the binary-rank
count used in Section~\ref{sec:binary-constraints}.
\end{remark}

\section{The trace-moment estimate}\label{sec:trace-proof}

This section proves Proposition~\ref{prop:trace}: under its hypotheses,
\[
 \bigl|\E\tr(RW)^{2p}\bigr|\le K_0^p(\delta\theta)^p.
\]
Besides estimating each term in the trace expansion, the proof must control
how many terms contribute. Section~\ref{sec:equality-patterns} first
averages the sampling variables and groups cyclic positions with equal
row labels. It bounds the number of the resulting equality partitions
and the loops in their graphs. Section~\ref{sec:binary-constraints}
then expands the entries of $R$ by cumulants. Lemma~\ref{lem:binary-constraints}
counts the row labels allowed by their binary constraints, and
Proposition~\ref{prop:entry-contribution} bounds the contribution of
each equality partition using Lemma~\ref{lem:cumulants}.
Finally, Section~\ref{sec:trace-summation} multiplies these
counts and bounds and sums over their parameters. Keeping track of
the lost free labels and cumulant blocks allows the rank hypothesis
to absorb the powers of the moment order in the counts.

\subsection{Equality patterns of the sampling variables}\label{sec:equality-patterns}

Recall that $R=P-\delta I_n$, $W=E-\theta I_n$, and
$\delta=r/n$. The entries $E_{ii}$ are independent
Bernoulli$(\theta)$ variables, independent of $D_1,D_2$, and the row
labels lie in $K=(\mathbb F_2)^m$, with $|K|=n$. Write
$w_i=W_{ii}=E_{ii}-\theta$. Matrix multiplication and cyclic
reindexing of the diagonal factors give the original trace expansion
\begin{equation}\label{eq:original-trace-expansion}
\begin{split}
 \E\tr(RW)^{2p}
   &=\sum_{i_1,\ldots,i_{2p}\in K}
       \E_{D_1,D_2}\!\left[
                 \prod_{j=1}^{2p}R_{i_j i_{j+1}}\right]
       \E_E\!\left[\prod_{j=1}^{2p}w_{i_j}\right],
 \qquad i_{2p+1}=i_1.
\end{split}
\end{equation}
In the direct matrix product the $j$th diagonal factor is
$w_{i_{j+1}}$; its product is the same as the one displayed.
The two expectations separate because the selectors are independent
of both sign diagonals.

Every word $(i_1,\ldots,i_{2p})$ determines one equality partition
$\pi$ of the position set $[2p]$: positions belong to the same
block if and only if their full row labels in $K$ coincide.
For a block $u\in\pi$, let $m_u=|u|$. Independence across distinct
selector coordinates then gives
\[
 \E_E\prod_{j=1}^{2p}w_{i_j}
      =\prod_{u\in\pi}\E w^{m_u},
 \qquad w=\zeta-\theta,\quad \zeta\sim\operatorname{Bernoulli}(\theta).
\]
A singleton block contributes $\E w=0$, so only partitions with
\(m_u\ge2\) remain. Let $v=|\pi|$ be the number of blocks.
Since $\sum_u m_u=2p$, it follows that $1\le v\le p$. Write
\[
 v=p-s,\quad 0\le s\le p-1.
\]
Thus $s$ counts the loss of blocks relative to a partition into $p$
pairs. To display the regrouping explicitly, let $u_j$ be the block
containing position $j$, with $u_{2p+1}=u_1$. Choosing a word with
equality partition $\pi$ is equivalent to assigning distinct labels
to its blocks, so \eqref{eq:original-trace-expansion} becomes
\[
 \E\tr(RW)^{2p}
 =\sum_{\substack{\pi\in\mathcal P([2p])\\ |u|\ge2\ (u\in\pi)}}
   \left(\prod_{u\in\pi}\E w^{m_u}\right)
   \sum_{\substack{\iota:\pi\to K\\ \iota\ \mathrm{injective}}}
   \E_{D_1,D_2}\!\left[
      \prod_{j=1}^{2p}R_{\iota(u_j),\iota(u_{j+1})}\right].
\]
Here $\mathcal P([2p])$ denotes the set of all partitions of $[2p]$.
This is the standard decomposition of a graph sum into injective
sums over quotient graphs; see Male
\cite[Definition~2.5 and Eq.~(2.2)]{male_traffic}.
This finite combinatorial identity is used without assumptions of
traffic independence.
The outer partitions are counted in this subsection, and
the inner entry sum is estimated in the next.

Construct the quotient multigraph of $\pi$ by taking its blocks as vertices
and, for each $j\in[2p]$, placing an edge occurrence between the
blocks containing $j$ and $j+1$, with cyclic indexing.
Edges are treated as undirected because $R$ is symmetric, but
their occurrences remain distinct. A position contributes one
incoming and one outgoing half-edge, including when these form a
loop. The cyclic word visits every block and connects each block to
the next, so the quotient graph is connected. It has \(2p\) edge
occurrences and degrees \(2m_u\ge4\). Since
$\sum_u\deg u=2\sum_u m_u=4p$ and $v=p-s$, it follows that
\begin{equation}\label{eq:degree-excess}
                         \sum_u(\deg u-4)=4s.
\end{equation}
Unlike the unrestricted contraction in
Definition~\ref{def:graph-contraction}, the sum here is restricted
to distinct full row labels for its \(v\) vertices.
Let \(t\) count the vertices incident to a
\emph{nonloop} unordered edge of odd multiplicity, and let \(\ell\)
count loop occurrences. In particular, $t$ counts vertices, not odd
edges; a vertex is counted once even if several odd-multiplicity
edges meet there. Since $t\le v=p-s$, it follows that \(s+t\le p\).
For a centered Bernoulli variable \(w\) and an integer \(b\ge2\),
its two possible values give
\[
 \E w^b=\theta(1-\theta)^b+(1-\theta)(-\theta)^b.
\]
Taking absolute values and factoring out $\theta(1-\theta)$ yields
\begin{equation}\label{eq:selector-moment}
 |\E w^b|
 \le\theta(1-\theta)\big((1-\theta)^{b-1}+\theta^{b-1}\big)
 \le\theta(1-\theta)\le\theta.
\end{equation}
The bounds $(1-\theta)^{b-1}\le1-\theta$ and
$\theta^{b-1}\le\theta$ hold for every $0<\theta<1$.
Thus the selector coefficient satisfies
\[
 \left|\prod_{u\in\pi}\E w^{m_u}\right|
 \le\prod_{u\in\pi}\theta=\theta^v.
\]
These are ordinary moments of equality blocks, not selector cumulants.

For example, take $p=2$ and distinct labels $a,b$. The word
$(a,b,a,b)$ has blocks $\{1,3\},\{2,4\}$ and gives four parallel
edges between its two degree-four vertices. The word $(a,a,b,b)$
has blocks $\{1,2\},\{3,4\}$ and gives two connecting edges and one
loop at each vertex. Both have $s=t=0$, whereas their loop counts
are respectively zero and two. These examples also explain why a
label appearing twice produces degree four, not degree two.

The parameters $s$ and $t$ will control the graph count. The first
controls the total degree above four by \eqref{eq:degree-excess};
the second locates vertices where nonloop occurrences cannot all be
paired with parallel copies. Section~\ref{sec:binary-constraints}
will relate those vertices to larger cumulant blocks or binary
constraints in Lemma~\ref{lem:binary-constraints}. A
bound on $\ell$ is also needed there, to count pairings
of loop occurrences. The counting class is specified next, independently
of any row-label assignment.

\begin{definition}[Equality-partition counting class]
\label{def:equality-class}
For $p\ge2$, $0\le s\le p-1$ and $0\le t\le p-s$,
let $\mathcal Q_{p,s,t}$ be the set of partitions $\pi$ of the
fixed position set $[2p]$ with exactly $p-s$ blocks, every block of
size at least two, and exactly $t$ vertices incident to an unordered
nonloop edge of odd multiplicity in the quotient graph just
defined. The positions retain their numbers: partitions related
by rotating or reversing the cycle are not identified unless they
are the same set partition. Put $N_{p,s,t}=|\mathcal Q_{p,s,t}|$
and $L=4p+1$. For each $\pi$ in this class, $\ell$ denotes the
number of its loop occurrences.
\end{definition}

The loop bound in
\eqref{eq:equality-count} holds for every individual graph,
whereas its bound on $N_{p,s,t}$ counts the entire class.
In particular, \(N_{p,s,t}\) counts partitions before any
assignment of row labels.
The abbreviation $L=4p+1$ collects the factors depending on the
moment order. The useful feature of the count is its dependence on
$s+t$: its upper bound is $8p\,3^p$ times $L^{74(s+t)}$, so
the additional powers of $L$ depend only on $s+t$. This dependence will
allow the later dimension factors to compensate for the graph count.

\begin{lemma}\label{lem:count}
For the class in Definition~\ref{def:equality-class}, each
$\pi\in\mathcal Q_{p,s,t}$ obeys the
following loop bound, and the counting class obeys the size bound:
\begin{equation}\label{eq:equality-count}
 \ell\le4t+12s+2,\qquad
                N_{p,s,t}\le8p\,3^p L^{74(s+t)}.
\end{equation}
\end{lemma}
\begin{proof}
The graph count uses the restricted local structure at vertices
of degree four whose nonloop multiplicities are even. This structure
first gives a bound on the loops. Doubled paths are then compressed,
and the retained graphs and their path lengths are counted, followed by
the traversals that recover the numbered cyclic positions.

\paragraph{The loop bound.}
Mark all odd-incident vertices and all vertices of degree greater
than four, and call their union \(B\). Every degree is even, so
each vertex of degree greater than four contributes at least two
to the total excess $4s$ in \eqref{eq:degree-excess}. There are
therefore at most \(2s\) such vertices. Moreover,
\[
 \sum_{u\in B}\deg u
 =4|B|+\sum_{u\in B}(\deg u-4)
 \le4(t+2s)+4s.
\]
Consequently,
\begin{equation}\label{eq:marked-degree}
 |B|\le t+2s,\qquad D_B:=\sum_{u\in B}\deg u\le4t+12s.
\end{equation}
Outside \(B\), degree is four and all nonloop multiplicities are
even. To list the possibilities, let $\ell_u$ be the number of loops
at such a vertex $u$, and let $a_{uz}$ count its edge occurrences
to another vertex $z$. A loop contributes two to the degree. In a
graph with more than one vertex, connectivity gives
\[
 4=2\ell_u+\sum_{z\ne u}a_{uz},\qquad
 a_{uz}\in2\mathbb Z_{\ge0},\qquad
 \sum_{z\ne u}a_{uz}>0.
\]
Hence $\ell_u$ is zero or one. If it is zero, the positive
nonloop multiplicities sum to four; if it is one, they sum to two.
The possibilities are exactly
\[
\begin{array}{c|c|c}
 \text{type} & \ell_u & \text{positive nonloop multiplicities}\\ \hline
 \text{internal} &0&(2,2)\\
 \text{loop terminal} &1&(2)\\
 \text{four-edge terminal} &0&(4)
\end{array}
\]
The internal type has two distinct neighbors. Each terminal type
has one neighbor.

Suppose first that \(B\ne\varnothing\). A four-edge terminal is directly adjacent
to \(B\). Otherwise its neighbor would also have degree four,
with all four edges returning to the terminal, giving an isolated
two-vertex component.

Starting instead at a loop terminal, follow
its doubled edge. At each internal vertex there is just one other
neighbor, so this doubled path continues without branching. If it
ended at another loop terminal, the path would be an entire
component disjoint from $B$; it cannot close into a cycle through
internal vertices, each of which already has its two neighbors.
It must therefore reach $B$.

Two terminal paths cannot merge at an internal vertex and then
continue to $B$: that would require a third neighbor. They therefore
use disjoint half-edges at $B$, at least two per terminal. Let $T$
be the number of terminal vertices and $\ell_B$ the number of
original loops at vertices in $B$. Each of these loops also uses
two half-edges at $B$, distinct from the terminal attachments. Thus
\[
 2T+2\ell_B\le D_B,\qquad T\le D_B/2.
\]
Every loop outside $B$ is the single loop of a loop terminal.
Consequently,
\[
 \ell=\ell_B+\#\{\text{loop terminals}\}
 \le\ell_B+T\le D_B.
\]

If \(B=\varnothing\), a four-edge terminal forces the whole graph
to be the two-vertex graph with four parallel edges, by the same
degree argument. Otherwise, in a graph with more than one vertex,
every nonloop edge group is doubled. Replace each such group by a
single edge and delete the loops. The resulting graph is connected,
with degree two at internal vertices and degree one at loop terminals.
Following edges can only close into a cycle or end in a path, since
there is no vertex at which to branch. Restoring the doubled edges
and loops gives a doubled cycle or a doubled path with one loop at
each endpoint. In the one-vertex case, degree four forces exactly
two loops. In all cases $\ell\le2$ when $B=\varnothing$.
Together with \eqref{eq:marked-degree}, this proves
the loop bound in \eqref{eq:equality-count}.

\paragraph{Counting the retained graphs.}
For the count, put $a=s+t$ and first suppose $a\ge1$.
Then $B\ne\varnothing$: either $t>0$ or the degree excess $4s$
is positive.
The count starts with the possible retained graphs and their expansions.
For each expanded graph, the next step counts the Euler transition systems
that recover the partition of the cyclic positions. These encode
closed traversals by pairing the incident half-edges at each vertex;
see Fleischner, Genest, and Jackson
\cite[Section~1]{fleischner_genest_jackson}.

Retain \(B\) and all terminal vertices, and contract maximal chains
of internal doubled-path vertices into distinguished doubled links.
Only their lengths can vary along these chains; each removed vertex
has degree four and the same two-neighbor structure.
A link may return to the same retained vertex or be parallel to
another edge; its paired half-edge ports must be retained.
Direct doubled edges may be given length zero. An isolated cycle
of removed vertices would be a component disjoint from the nonempty
set $B$, so connectivity excludes it.

Here a half-edge is an incidence of one edge at one endpoint;
a loop has two half-edges at its vertex. A port is the pair of
half-edges belonging to the two copies of a doubled path at one
retained endpoint. A returning path has two distinct ports at the
same vertex. Its two contracted edges are loops, but each loop
joins one half-edge of one port to one half-edge of the other.
These loops created by contraction are distinct from the original
loop occurrences counted by $\ell$.
The length records the number of removed internal vertices,
not the number of edges. Expanding a link of length $j$ therefore
replaces its two edges by $j+1$ doubled path segments.

The $p=2$ words above illustrate the simplest equality graphs.
The following example includes both a returning link and a link
parallel to other edges. It will illustrate the two counting steps:
encoding the graph by its links, then recovering the cyclic positions
from a traversal. Consider the word
\begin{equation}\label{eq:encoding-example-word}
 (b,x,y,b,y,x,b,z,c,b,c,z,b,c).
\end{equation}
Its last letter is followed by its first. Here $p=7$ and the five
vertex multiplicities are $m_b=5$, $m_c=3$, and
$m_x=m_y=m_z=2$. Thus $s=2$, all degrees are at least four, and
the graph is connected.

There are two copies of each segment of
$b$--$x$--$y$--$b$, two copies of each segment of
$b$--$z$--$c$, and four direct $b$--$c$ edges. Every nonloop
multiplicity is even, so $t=0$ and $B=\{b,c\}$.
The vertices $x,y,z$ are internal: each has two distinct neighbors
joined to it by doubled edges and has no loop.
Suppressing $x,y$ gives a returning link of length two at $b$;
suppressing $z$ gives a length-one link between $b,c$, parallel to
the four direct edges. Figure~\ref{fig:p7-contraction} shows the two
graphs. In the contracted graph, the half-edges of the two ports at
$b$ are named $\{b_1,b_2\}$ and $\{b_3,b_4\}$; the full labeling
and traversal of this example are given after the proof.

\input{figures/p7_contraction.tex}

\begin{samepage}
Now return to the arbitrary graph with $a=s+t\ge1$.
All $T$ terminal vertices are retained, and the bound
$T\le D_B/2$ has already been established. Contraction preserves the degree of every retained
vertex: the two incidences of a doubled path at each endpoint are
replaced by the two incidences of its link. This also holds at a
returning link, whose two ports remain distinct. Each terminal has
degree four.
The retained graph has \(q\) vertices and total degree \(D\), with
\begin{equation}\label{eq:retained-graph}
 q\le |B|+D_B/2\le8a,\qquad
 D=D_B+4T\le3D_B\le36a,\qquad q\le p,\quad D\le4p.
\end{equation}
\end{samepage}
The bounds $q\le p$ and $D\le4p$ hold because only vertices were
removed, while all retained degrees were preserved.
Temporarily label its \(q\) vertices and distinguish their half-edges.
The following data give a surjective encoding of its possible
expansions: every original graph has a code. A graph may have
several codes, which is harmless for an upper bound.

\begin{enumerate}
\item Choose the number of retained vertices and their degrees.
Since $1\le q\le p$, $q\le8a$, and every degree is at most $4p$,
\[
 \#\{\text{choices of }q\}\le p\le L^a,\qquad
 \#\{\text{degree vectors for fixed }q\}\le(4p+1)^q\le L^{8a}.
\]
Retain only even degrees at least four and totals obeying the
bounds on $D$ above. The degree vector determines $D$ and a list
of $D$ distinct half-edges, grouped by vertex.

\item Pair the half-edges to specify the edges of the retained
multigraph. Pairing the first unpaired half-edge leaves $D-1$
choices, then $D-3$, and so on. Therefore
\[
 \#\{\text{edge pairings}\}=(D-1)!!
 =\prod_{j=1}^{D/2}(2j-1)\le D^{D/2}\le L^{18a}.
\]
The last inequality uses $D\le4p<L$ and $D/2\le18a$.
Pairing two half-edges at the same vertex is allowed and produces a loop.

\item Designate disjoint pairs of half-edges at each vertex as ports.
For an upper bound, give each half-edge either no mate or one of
the $D$ half-edge names as a mate. This gives
\[
 \#\{\text{port choices}\}\le(D+1)^D\le L^{36a},
\]
using $D+1\le L$ and $D\le36a$. Retain only mutual pairings of
distinct half-edges at the same vertex, and require each port to
be edge-paired to one other whole port. Explicitly, after naming
the half-edges of the two ports, the link has the form
\[
 \underbrace{\{h_1,h_2\},\ \{k_1,k_2\}}_{\text{two ports}},
 \qquad
 \underbrace{\{h_1,k_1\},\ \{h_2,k_2\}}_{\text{two edge pairs}}.
\]
The four half-edges remain distinct even if the ports are at the
same vertex. The ports identify the link even when other edges
have the same endpoints.

\item Choose the length of each link and expand its internal
vertices in path order. A link uses four half-edges, so there are
at most $D/4\le9a$ links. A length counts removed vertices and is
at most $p$, since the original graph has $v\le p$ vertices. Hence
\[
 \#\{\text{length assignments}\}\le(p+1)^{D/4}\le L^{9a}.
\]
Keep only expansions with the prescribed vertex and edge counts.
\end{enumerate}
Multiplying the five upper bounds gives
\[
 \underbrace{L^a}_{\text{vertex count}}\,
 \underbrace{L^{8a}}_{\text{degrees}}\,
 \underbrace{L^{18a}}_{\text{edges}}\,
 \underbrace{L^{36a}}_{\text{ports}}\,
 \underbrace{L^{9a}}_{\text{lengths}}
 =L^{72a}
\]
as a bound on the number of graph encodings.

For decoding, give the retained vertices the names $1,\ldots,q$
and number their incident half-edges in the order of the degree
vector. The edge and port pairings above then specify every retained
edge and every link, including returning links.

To make the expansion deterministic, order the links by their
smallest half-edge labels. On each link, start at the port with
the smaller label and insert the prescribed number of internal
vertices in path order. Give the new vertices and half-edges fresh
names in this order. Thus the code fixes their names; no further
labeling choice is needed. Undesignated half-edges keep their original
edge pairings, even when they share endpoints with a designated link.

Every original graph yields such data by its specified contractions.
Conversely, valid data reconstruct its graph with distinguished
half-edges. Invalid and duplicate codes enlarge the count. Core
labels are auxiliary; no separately labeled set of original vertices
is chosen.

\paragraph{Recovering the cyclic positions.}
Now fix one expanded graph reconstructed by these data.
Its edges do not yet specify the equality partition: different
traversals of the same graph can place its vertices at different
cyclic positions. An Euler circuit, meaning a closed traversal
using every edge occurrence once, pairs arriving and departing
half-edges at each vertex. These local transition pairs specify
how the traversal continues through a vertex; the earlier edge
pairing specified how it passes between vertices.
For degree \(2m_u\), pairing the first half-edge has $2m_u-1$
choices, the next unpaired half-edge has $2m_u-3$, and so on.
Thus the number of local transition pairings is
\[
 (2m_u-1)!!=3\prod_{j=3}^{m_u}(2j-1)
                                  \le3(4p)^{m_u-2}.
\]
Here $m_u\ge2$ and each factor $2j-1$ is at most $4p$;
when $m_u=2$, the product on the right is empty and equals one.
Multiplying over the $v$ vertices gives
\[
 \prod_u(2m_u-1)!!
 \le3^v(4p)^{\sum_u(m_u-2)}
 =3^v(4p)^{2s}\le3^pL^{2s},
\]
where $\sum_u(m_u-2)=2p-2v=2s$ and $v\le p$.

Choose a starting directed edge in at most
\(4p\) ways: there are $2p$ edge occurrences and two directions
for each, including the two choices of departure half-edge for a
loop. To decode, write $h_j$ for the departing half-edge at step $j$
and $\bar h_j$ for the other half-edge of that edge occurrence.
The rule for continuing the traversal is
\[
 h_j\ \xrightarrow{\ \text{edge pairing}\ }\ \bar h_j
     \ \xrightarrow{\ \text{local pairing}\ }\ h_{j+1}.
\]
Both pairings are fixed, so the initial directed edge determines
every subsequent step. Local pairings can divide the edge occurrences
among several circuits; retain only systems that traverse all $2p$
edge occurrences in one circuit.

For such a system, let $u_j$ be the vertex at which $h_j$ departs.
The recovered equality partition is exactly
\[
 \bigl\{\{j\in[2p]:u_j=u\}:u\text{ is a vertex of the expanded graph}\bigr\}.
\]
Every original equality partition produces one of these systems:
its cyclic positions specify the successive edge occurrences, and
each visit pairs its arriving half-edge with its departing half-edge.
The edge at position 1 supplies the starting direction. Transferring
these pairs to an isomorphic decoded graph therefore recovers the
original partition.

Renaming the auxiliary vertices preserves every equality $u_j=u_k$,
so it does not change this partition. Their names and their positions
in the cyclic word are already fixed by the code and traversal; no
additional vertex permutation or placement factor is required.
Combining the graph and traversal counts gives
\[
 L^{72a}\,3^pL^{2s}\,(4p)
 \le4p\,3^pL^{74a},\qquad s\le a.
\]

For \(a=0\), the underlying graph is a doubled cycle or a doubled
path with a loop at each endpoint, as in the $B=\varnothing$
case of the loop bound. Since $s=0$, there are $v=p\ge2$ vertices,
so the one-vertex case does not arise. With $p$ vertices fixed,
there are at most these two graphs. Each vertex has exactly three
local transition pairings, giving
\[
 N_{p,0,0}\le
 \underbrace{2}_{\text{graphs}}\,
 \underbrace{4p}_{\text{starting directions}}\,
 \underbrace{3^p}_{\text{local pairings}}
 =8p\,3^p.
\]
This proves \eqref{eq:equality-count}.
\end{proof}

\paragraph{Decoding the example.}
Return to the word in \eqref{eq:encoding-example-word} and its
contraction in Figure~\ref{fig:p7-contraction}.
This example makes the encoding behind the bound on $N_{p,s,t}$
in Lemma~\ref{lem:count} concrete. Its original graph has no loops
or terminal vertices; the terminal-attachment argument for the loop
bound is a separate part of the proof. The correspondence with the
counting parameters is
\begin{center}
\begin{tabular}{p{0.18\textwidth}|p{0.28\textwidth}|p{0.42\textwidth}}
Object & In this example & Role in the proof\\ \hline
Equality class & $p=7$, $v=5$, $s=2$, $t=0$
 & Fourteen positions and five letters give $s=p-v=2$;
 all nonloop multiplicities are even, so the partition lies in
 $\mathcal Q_{7,2,0}$.\\
Marked vertices & $B=\{b,c\}$
 & These are the only vertices of degree greater than four;
 they remain when the internal vertices are removed.\\
Retained graph & $q=2$, $D=10+6=16$
 & These are the vertex and half-edge counts used in
 \eqref{eq:retained-graph} and the graph encoding.\\
Link lengths & $2$ and $1$
 & Expanding the returning link inserts $x,y$;
 expanding the other link inserts $z$.
\end{tabular}
\end{center}
Suppressing $x,y,z$ leaves $q=2$ retained vertices of degrees
ten and six. Label their half-edges $b_1,\ldots,b_{10}$ and
$c_1,\ldots,c_6$. One valid retained code has the edge pairs
\begin{equation}\label{eq:encoding-example-edges}
\begin{gathered}
 \{b_1,b_3\},\ \{b_2,b_4\},\
 \{b_5,c_1\},\ \{b_6,c_2\},\\
 \{b_7,c_3\},\ \{b_8,c_4\},\
 \{b_9,c_5\},\ \{b_{10},c_6\}.
\end{gathered}
\end{equation}
Designate $\{b_1,b_2\}$ and $\{b_3,b_4\}$ as the two ports
of a length-two returning link. Designate $\{b_5,b_6\}$ and
$\{c_1,c_2\}$ as the ports of a length-one link between $b,c$.
The four remaining edges have no designated ports. The first
link expands to the doubled path $b$--$x$--$y$--$b$,
and the second to $b$--$z$--$c$, parallel to the four direct
$b$--$c$ edges. A returning link is therefore distinguished by
its two ports at $b$, while a parallel link is distinguished
from the other edges by its designated half-edges.

\begin{samepage}
For actual transition data after expansion, name the fourteen
edge occurrences $e_1,\ldots,e_{14}$ in their traversal order in
\eqref{eq:encoding-example-word}; repeated endpoint pairs use
different copies. At each vertex the arriving/departing edge
pairs are
\begin{equation}\label{eq:encoding-example-transitions}
\begin{aligned}
 b &: (e_{14},e_1),(e_3,e_4),(e_6,e_7),
                    (e_9,e_{10}),(e_{12},e_{13}),\\
 c &: (e_8,e_9),(e_{10},e_{11}),(e_{13},e_{14}),\\
 x &: (e_1,e_2),(e_5,e_6),\qquad
 y : (e_2,e_3),(e_4,e_5),\\
 z &: (e_7,e_8),(e_{11},e_{12}).
\end{aligned}
\end{equation}
\end{samepage}
The two returning-link copies in
\eqref{eq:encoding-example-edges} expand along
$(e_1,e_2,e_3)$ and $(e_6,e_5,e_4)$, respectively.
The copies of the other link expand along $(e_7,e_8)$ and
$(e_{12},e_{11})$. The four undesignated edges become
$e_9,e_{10},e_{13},e_{14}$ in their listed order.
Start with $e_1$ directed from $b$ to $x$.
At $x$, the pair $(e_1,e_2)$ in
\eqref{eq:encoding-example-transitions} forces the next edge to be
$e_2$. At $y$, the pair $(e_2,e_3)$ selects $e_3$, which returns
to $b$. There the pair $(e_3,e_4)$ selects $e_4$. The first steps are
\[
 \underbrace{b}_{1}\xrightarrow{e_1}
 \underbrace{x}_{2}\xrightarrow{e_2}
 \underbrace{y}_{3}\xrightarrow{e_3}
 \underbrace{b}_{4}\xrightarrow{e_4}
 \underbrace{y}_{5}.
\]
The numbers beneath the vertices are positions in the cyclic word.
The return to $b$ already places positions $1,4$ in the same equality
block; the two visits to $y$ place positions $3,5$ in another.
Continuing to alternate edge traversals with the specified transitions
reconstructs the displayed word and hence the partition
\[
 \bigl\{\{1,4,7,10,13\},\{9,11,14\},
          \{2,6\},\{3,5\},\{8,12\}\bigr\}.
\]
All fourteen edge occurrences are traversed once. The example
belongs to $\mathcal Q_{7,2,0}$; the auxiliary names $b,c,x,y,z$
play no role in the resulting set partition.
This is the decoding used in the counting proof: once the graph
code, local transition pairs, and starting direction have been chosen,
the positions belonging to each vertex require no further choice.

\subsection{Cumulant partitions and binary constraints}\label{sec:binary-constraints}

Section~\ref{sec:equality-patterns} reduced the trace expansion to a
sum over equality partitions and injective row-label assignments.
Fix one equality partition and consider its contribution. There
are two tasks: count the cumulant partitions of its edge occurrences,
and count the row labels allowed by the XOR conditions of
Lemma~\ref{lem:cumulants}. The rank of these conditions will also
relate the original graph's odd-multiplicity edges to the sizes
of the cumulant blocks.

Fix $\pi\in\mathcal Q_{p,s,t}$ as in
Definition~\ref{def:equality-class}, with vertex set $\mathcal V=\pi$.
Recall that $v=|\mathcal V|=p-s$, that $t$ counts vertices incident
to odd-multiplicity nonloop edges, and that $\ell$ counts loop
occurrences. The notation $L=4p+1$ remains in force.
Let $u_j$ be the vertex containing position $j$, with $u_{2p+1}=u_1$.
The row labels lie in $K=(\F_2)^m$, where $n=2^m$; an injective
assignment $\iota:\mathcal V\to K$ gives distinct labels to distinct
vertices. Applying the moment formula \eqref{eq:moment-cumulant}
in Lemma~\ref{lem:cumulant-identities} gives
\begin{equation}\label{eq:entry-moment-expansion}
 \E_{D_1,D_2}\prod_{j=1}^{2p}R_{\iota(u_j),\iota(u_{j+1})}
   =\sum_{\rho\in\mathcal P([2p])}
      \prod_{C\in\rho}
        \cum\bigl(R_{\iota(u_j),\iota(u_{j+1})}:j\in C\bigr).
\end{equation}
The partition $\pi$ groups vertex positions and fixes which row
labels coincide; $\rho$ groups edge occurrences to form cumulants.
It does not identify the graph's vertices. Recall that
$R=P-\delta I_n$ has centered entries, including its diagonal
entries. Hence singleton blocks of $\rho$ vanish, including singleton
loops.

\begin{samepage}
A partition with no singleton blocks has
\[
                         b=p-d,\qquad 0\le d\le p-1
\]
blocks. Here $d$ measures the loss of cumulant blocks from the
partition into $p$ pairs, whereas $s$ measures the loss of vertex
blocks in $\pi$.
\end{samepage}

For each cumulant block, Lemma~\ref{lem:cumulants} requires the sum
of its endpoint labels to be zero in $K$. The following matrix
collects these equations. Its rank counts independent restrictions
on the row labels, rather than counting the equations with repetition.

\begin{definition}[Binary constraint matrix]\label{def:binary-constraint}
For an equality graph with vertex set $\mathcal V$ and successive
vertices $u_1,\ldots,u_{2p}$, with $u_{2p+1}=u_1$, and a
partition $\rho$ of its edge
occurrences, the binary constraint matrix has one column per vertex
and one row per block: the
sum of that block's endpoint incidence vectors over \(\F_2\).
Writing $\mathbf e_u$ for the coordinate vector of a graph vertex,
the row belonging to $C\in\rho$ is
\[
 \sum_{j\in C}(\mathbf e_{u_j}+\mathbf e_{u_{j+1}})
                       \quad\hbox{in }\F_2^{\mathcal V}.
\]
Repeated endpoints cancel in pairs; in particular a loop
contributes zero. Write $h$ for the rank of this matrix over $\F_2$.
\end{definition}

Some pair blocks must be removed while the labels are still
injective. Otherwise, counting all solutions of the binary system
would retain partitions that actually vanish for every labeling in
the trace expansion. The next lemma performs this removal and gives
the two consequences of the binary constraints used below.

\begin{lemma}\label{lem:binary-constraints}
Fix $\pi\in\mathcal Q_{p,s,t}$ and a partition $\rho$ of its edge
occurrences into $b=p-d$ blocks, all of size at least two. If a pair
block consists of adjacent nonparallel nonloop edges, or of a loop
and a nonloop edge, its cumulant vanishes for every injective
row-label assignment.

After excluding partitions with either kind of pair block, the XOR
conditions have exactly $n^{v-h}$ unrestricted row-label assignments,
and hence at most $n^{v-h}$ injective ones, where $v=p-s$ and $h$ is
as in Definition~\ref{def:binary-constraint}. Moreover,
\begin{equation}\label{eq:odd-incidence}
                              t\le12d+4h.
\end{equation}
The XOR conditions are necessary for a nonzero cumulant product;
no sufficiency is asserted.
\end{lemma}
\begin{proof}
For adjacent nonparallel edges $\{u,v\}$ and $\{u,w\}$, the
vertices $v,w$ are distinct. Their binary row is
\[
 (\mathbf e_u+\mathbf e_v)+(\mathbf e_u+\mathbf e_w)
       =\mathbf e_v+\mathbf e_w.
\]
Its XOR condition is $\iota(v)+\iota(w)=0$, or
$\iota(v)=\iota(w)$, which contradicts injectivity. A loop
contributes zero, so pairing it with a nonloop edge $\{v,w\}$
gives the same row and the same contradiction. In both cases
Lemma~\ref{lem:cumulants} makes the pair cumulant zero.

The remaining pair types are parallel nonloop occurrences,
disjoint nonloop edges, and two loops. The first and third types
have zero binary row. For disjoint edges $\{u,v\}$ and
$\{w,z\}$ the four endpoints are distinct, and the row is
\[
                  \mathbf e_u+\mathbf e_v+\mathbf e_w+\mathbf e_z.
\]
Thus every nonzero pair row has weight four, meaning four nonzero
coordinates.

Multiplying the row of a general block by the vertex labels gives
exactly its XOR condition \eqref{eq:xor-constraint}. For each of
the $m$ binary coordinates of $K=(\F_2)^m$, the matrix has rank
$h$ and its homogeneous system in $v$ unknowns has $2^{v-h}$
solutions. The coordinates can be chosen independently, giving
\[
                         (2^{v-h})^m=n^{v-h}.
\]
Imposing injectivity can only reduce this number.

To prove \eqref{eq:odd-incidence}, let $M$ count edge occurrences
in blocks of size at least three. Since there are $b=p-d$ blocks,
\[
 \sum_{C\in\rho}(|C|-2)=2p-2b=2d.
\]
Pair blocks contribute zero to this sum. For $j\ge3$,
$j\le3(j-2)$, so
\[
 M=\sum_{\substack{C\in\rho\\|C|\ge3}}|C|
   \le3\sum_{\substack{C\in\rho\\|C|\ge3}}(|C|-2)=6d.
\]
At most $2M\le12d$ vertices touch these occurrences.

Now consider an odd-incident vertex that touches no large-block
occurrence. Choose an incident nonloop edge of odd multiplicity.
All occurrences of that edge belong to pair blocks. They cannot
all be paired with parallel copies, because their number is odd.
At least one is paired with a different edge. The two forbidden
pair types have already been excluded, so that pair is disjoint.
This vertex therefore belongs to the support of a weight-four row.

Let $U$ be the union of the supports of these disjoint-pair rows.
Choose a basis from the rows themselves. It has at most $h$ rows,
since their span is contained in the full constraint row space.
Any coordinate that is zero in every basis row is zero in every
linear combination. Hence all disjoint-pair rows are supported
on the union of the basis supports, and
\[
                              |U|\le4h.
\]
Every odd-incident vertex either touches a large block or belongs
to $U$. This gives $t\le12d+4h$.
\end{proof}

The variables introduced so far record distinct reductions in the
number of free choices. Recall that $\kappa=n\theta$ is the mean
sample size and $\delta=r/n$. The following table summarizes the
roles of the four parameters; the exact factorization appears in
\eqref{eq:dimension-factor}.
\begin{center}
\begin{tabular}{c|p{0.43\textwidth}|p{0.36\textwidth}}
Parameter & What it counts & Factor or constraint used below\\ \hline
$s=p-v$ & Loss of selector equality blocks from the $p$-pair case
         & $\kappa^{-s}$\\
$d=p-b$ & Loss of entry-cumulant blocks from the $p$-pair case
         & $r^{-d}$\\
$h$ & Rank of the binary endpoint-constraint matrix
     & $n^{-h}$\\
$t$ & Vertices incident to odd-multiplicity nonloop edges
     & $t\le12d+4h$
\end{tabular}
\end{center}
The first two losses arise in different partitions. Identifying
selector positions reduces $v$, while putting more edge
occurrences in a cumulant reduces $b$. An independent binary
constraint removes one freely chosen label in $K$.
There is no separate dimensional factor for $t$; instead,
\eqref{eq:odd-incidence} forces every increase of $t$ to be
accounted for by $d$ or $h$. This is what allows the powers
of the moment order in the counts to be absorbed by the rank.

The partition count can now be combined with these dimension factors.
For fixed $\pi$, let $\mathcal R_{d,h}(\pi)$ denote the edge
partitions with $p-d$ blocks of size at least two, binary rank $h$,
and neither forbidden pair type from Lemma~\ref{lem:binary-constraints}.
This notation only groups the remaining terms of
\eqref{eq:entry-moment-expansion}; it does not require their
cumulants to be nonzero. Write
\[
 c(\rho)=\prod_{C\in\rho}(2|C|)^{12|C|},\qquad K_2=4^{24}.
\]
The weight $c(\rho)$ collects the order-dependent constants in
\eqref{eq:cumulant-bound}. Each pair block contributes $K_2$.
Counting with this weight keeps the partition count and its
cumulant constants in one estimate.

\begin{proposition}
\label{prop:entry-contribution}
Fix $\pi\in\mathcal Q_{p,s,t}$, and let $d,h$ be nonnegative
integers. With the notation above,
\begin{equation}\label{eq:entry-partition-count}
 \sum_{\rho\in\mathcal R_{d,h}(\pi)}c(\rho)
       \le (3K_2)^p L^{84d+12h+9s+2t+1}.
\end{equation}
The contribution of these partitions to the trace expansion,
after summing absolute values over their injective row-label
assignments and multiplying by the absolute selector coefficient,
is at most
\begin{equation}\label{eq:fixed-equality-contribution}
 (3K_2)^p L^{84d+12h+9s+2t+1}
       (\delta\theta)^p\kappa^{-s}r^{-d}n^{-h}.
\end{equation}
If $\mathcal R_{d,h}(\pi)$ is nonempty, its parameters satisfy
$t\le12d+4h$.
\end{proposition}
\begin{proof}
If $\mathcal R_{d,h}(\pi)$ is empty, both bounds are immediate.
The count starts with the large blocks, then the disjoint pairs, loop
pairs, and parallel nonloop pairs. These are all the possibilities
left by Lemma~\ref{lem:binary-constraints}. Throughout the count,
edge occurrences retain their position numbers in $[2p]$.

\paragraph{Large blocks.}
As in the preceding proof, at most $M\le6d$ occurrences lie in
blocks of size at least three. A subset of at most $6d$ positions
can be recorded as its increasing list padded with zeros to length
$6d$, giving at most $(2p+1)^{6d}$ choices. A partition of the
chosen positions can be encoded by assigning each position a block
label from $[2p]$, giving at most $(2p)^{6d}$ choices. Thus there
are at most
\[
                 (2p+1)^{6d}(2p)^{6d}\le L^{12d}
\]
ways to choose and partition the large-block occurrences. These codes
may overcount; only blocks of size at least three with the
required total excess are retained. Their cumulant constants obey
\[
 \prod_{\substack{C\in\rho\\|C|\ge3}}(2|C|)^{12|C|}
       \le(4p)^{12M}\le L^{72d}.
\]
There are at most $p$ pair blocks, whose constants together are at
most $K_2^p$. When $d=0$ the large-block set is empty, and both
large-block counting factors equal one.

\paragraph{Disjoint pairs.}
Let $U$ again be the union of the supports of the disjoint-pair
rows. The preceding proof gives $|U|\le4h$. Listing these vertices
in a fixed order and padding the list gives at most
$(v+1)^{4h}\le L^{4h}$ choices for $U$.
Every occurrence in a disjoint pair has both endpoints in $U$.
Each such occurrence uses two incidences there, so the number
available is at most
\[
 \frac12\sum_{u\in U}\deg u
 =2|U|+\frac12\sum_{u\in U}(\deg u-4)
 \le2|U|+2s\le8h+2s.
\]
Here all degree excesses are nonnegative, and their sum over the
whole graph is $4s$ by \eqref{eq:degree-excess}.
Choosing a subset of the available occurrences costs at most
$2^{8h+2s}$. If the chosen subset has $2j$ occurrences, pairing it
costs $(2j-1)!!\le(2p)^j\le(2p)^{4h+s}$; the empty subset
has one pairing. Therefore
\[
 2^{8h+2s}(2p)^{4h+s}\le L^{8h+2s}.
\]
The last inequality uses $2\sqrt{2p}\le4p+1=L$.
Including the choice of $U$ gives
$L^{12h+2s}$. Pairings are discarded if they are not disjoint, if their
support is not the chosen $U$, or if their resulting full binary
rank is not $h$.

\paragraph{Loop pairs and parallel pairs.}
Once the large blocks and disjoint pairs have been chosen, every
remaining loop must be paired with another loop. If their number
is odd there is no such partition. Otherwise their pairing count
is at most $(2p)^{\ell/2}$, since their number is at most the
original loop count $\ell$. The loop bound
$\ell\le4t+12s+2$ from Lemma~\ref{lem:count} gives
\begin{equation}\label{eq:loop-pair-count}
                  (2p)^{\ell/2}\le L^{2t+6s+1}.
\end{equation}

The remaining nonloop occurrences can only pair with parallel
copies. For each unordered nonloop edge $e$, write $a_e$ for its
remaining multiplicity. An odd $a_e$ gives no valid pairing.
For positive even $a_e$, the number of pairings is $(a_e-1)!!$.
It equals one for $a_e=2$ and three for $a_e=4$; for $a_e\ge6$,
\[
 (a_e-1)!!=3\cdot5\cdot7\cdots(a_e-1)
       \le3(4p)^{(a_e-4)/2}.
\]
Thus the uniform bound is $3(4p)^{(a_e-4)_+/2}$, where
$(x)_+=\max\{x,0\}$. There are at most $p$ positive groups,
because each contains at least two of the $2p$ occurrences.

The original degree excess controls the exponents, even
though some occurrences have already been removed. At a vertex
$u$, if no remaining group has multiplicity above four, the sum
of $(a_e-4)_+$ over incident groups is zero. Otherwise their total
multiplicity is at most $\deg u$, and subtracting four for each
such group subtracts at least four. In either case,
\[
 \sum_{e\ni u}(a_e-4)_+\le\deg u-4.
\]
Every nonloop edge has two endpoints. Summing this inequality over
vertices and using \eqref{eq:degree-excess} gives
\[
 2\sum_e(a_e-4)_+\le\sum_u(\deg u-4)=4s,
 \qquad \sum_e(a_e-4)_+\le2s.
\]
All parallel-pair choices therefore cost at most
\[
 \prod_{e:a_e>0}(a_e-1)!!
       \le3^p(4p)^{\frac12\sum_e(a_e-4)_+}\le3^pL^s.
\]
Combining the large-block count, cumulant constants, and the three
pair counts yields
\[
 L^{12d}\,L^{72d}K_2^p\,L^{12h+2s}\,
 L^{2t+6s+1}\,3^pL^s
       =(3K_2)^pL^{84d+12h+9s+2t+1},
\]
which proves \eqref{eq:entry-partition-count}. The counts permit
incompatible choices, but every partition in
$\mathcal R_{d,h}(\pi)$ is covered. No common sign is assumed
for cumulant terms.

\paragraph{Row labels and dimension factors.}
For a block of size $q\ge2$, centering and multilinearity transfer
\eqref{eq:cumulant-bound} to the entries of $R$ as
\[
 \left|\cum(R_{i_1j_1},\ldots,R_{i_qj_q})\right|
       \le(2q)^{12q}r\,n^{-q}.
\]
There are $b=p-d$ blocks and $2p$ occurrences in total. Their
cumulant product is therefore bounded in absolute value by
$c(\rho)r^{p-d}n^{-2p}$, and vanishes unless every XOR condition
holds. After the forbidden pair
types have been removed, Lemma~\ref{lem:binary-constraints} allows
injectivity to be relaxed in this absolute upper bound, with at
most $n^{v-h}$ labels. The selector coefficient from
\eqref{eq:selector-moment} is at most $\theta^v$.
Their combined dimension factor is exactly
\begin{equation}\label{eq:dimension-factor}
 \theta^v r^{p-d}n^{-2p}n^{v-h}
            =(\delta\theta)^p\kappa^{-s}r^{-d}n^{-h}.
\end{equation}
To see the substitutions, use $v=p-s$, $\delta=r/n$, and
$\kappa=n\theta$:
\[
 \theta^{p-s}r^{p-d}n^{-p-s-h}
       =\left(\frac{r\theta}{n}\right)^p
          (n\theta)^{-s}r^{-d}n^{-h}.
\]
Multiplying \eqref{eq:entry-partition-count} by this factor proves
\eqref{eq:fixed-equality-contribution}. The parameter constraint is
\eqref{eq:odd-incidence} in Lemma~\ref{lem:binary-constraints}.
\end{proof}

For comparison, the case $v=b=p$ and $h=0$ contributes
$\theta^pr^pn^{-p}=(\delta\theta)^p$ before its combinatorial
coefficient. Reducing $v$ by one removes one selector factor
$\theta$ and one free-label factor $n$, hence costs $\kappa^{-1}$.
Reducing $b$ by one removes one factor $r$; increasing $h$ by one
removes one factor $n$. This gives every exponent in
\eqref{eq:dimension-factor} directly.

\subsection{Summation of the trace expansion}\label{sec:trace-summation}

The counts from Sections~\ref{sec:equality-patterns}
and~\ref{sec:binary-constraints} can now be combined. The constraint $t\le12d+4h$
will eliminate $t$, leaving three sums whose dimension factors
are controlled by the rank hypothesis in Proposition~\ref{prop:trace}.

\begin{proof}[Completion of the proof of Proposition~\ref{prop:trace}]
Recall that $A=RW$, $\delta=r/n$, and
\[
 L=4p+1,\qquad \kappa=n\theta,\qquad
 K_2=4^{24},\qquad K_0=576K_2.
\]
Here $s=p-v$ and $d=p-b$ measure the losses of vertex blocks
and cumulant blocks, respectively; $h$ is the binary constraint
rank, and $t$ counts the odd-incident vertices.

For each $\pi\in\mathcal Q_{p,s,t}$, apply
Proposition~\ref{prop:entry-contribution}, whose bound
\eqref{eq:fixed-equality-contribution} combines
\eqref{eq:entry-partition-count} with \eqref{eq:dimension-factor}.
Multiply this by the equality-partition count \eqref{eq:equality-count}
from Lemma~\ref{lem:count}. The powers of $L$ and the remaining
combinatorial factors combine as follows:
\[
 \begin{aligned}
 74(s+t)+(84d+12h+9s+2t+1)
       &=83s+76t+84d+12h+1,\\
 8p\,3^p(3K_2)^p&=8p(9K_2)^p.
 \end{aligned}
\]
Taking the last factor $L$ outside the sum, the finite expansion
and the triangle inequalities in Proposition~\ref{prop:entry-contribution}
bound \(|\E\tr A^{2p}|\) by
\begin{equation}\label{eq:trace-sum}
 8pL(9K_2)^p(\delta\theta)^p
  \sum_{s,t,d,h}L^{83s+76t+84d+12h}\kappa^{-s}r^{-d}n^{-h},
\end{equation}
where the sum runs over the feasible equality-partition parameters
$s,t$ and cumulant-partition parameters $d,h$. Since $v=p-s$
and $b=p-d$ lie between one and $p$, the block losses lie between
zero and $p-1$. Also, $t\le v$, and the binary matrix has $b$ rows
and $v$ columns, so $h\le\min\{v,b\}$. Together with
Lemma~\ref{lem:binary-constraints}, these observations give
\[
 \begin{gathered}
  0\le s,d\le p-1,\qquad 0\le t\le p-s,\\
  0\le h\le\min\{p-s,p-d\},\qquad t\le12d+4h.
 \end{gathered}
\]

For fixed $s,d,h$, the feasible $t$-range may be enlarged to all
integers from zero to $T=12d+4h$, because every term in the upper
bound is nonnegative. To sum these powers, factor out the largest
one and use $L\ge9$:
\[
 \sum_{t=0}^{T}L^{76t}
   =L^{76T}\sum_{j=0}^{T}L^{-76j}
   \le\frac{L^{76T}}{1-L^{-76}}
   \le2L^{76T}.
\]
Since $76T=912d+304h$, this gives
\[
 \sum_{t=0}^{12d+4h}L^{76t}\le2L^{912d+304h}.
\]
Thus the factor $8pL$ in \eqref{eq:trace-sum} becomes $16pL$,
and the remaining powers of $L$ have exponents
$83s+(84+912)d+(12+304)h=83s+996d+316h$.

The hypotheses give $\kappa\ge r$, and $n\ge r$ because the
frame has $r$ orthonormal columns in $\R^n$. Since $s,d,h$ are
nonnegative and each of $83,996,316$ is at most $1000$,
\[
 \frac{L^{83s+996d+316h}}{\kappa^s r^d n^h}
   \le\frac{L^{83s+996d+316h}}{r^{s+d+h}}
   \le\left(\frac{L^{1000}}r\right)^{s+d+h}.
\]
All three parameter ranges can now be extended to nonnegative
integers, again only increasing the upper bound. Consequently
\eqref{eq:trace-sum} is at most
\[
 16pL(9K_2)^p(\delta\theta)^p
          \sum_{s,d,h\ge0}\left(\frac{L^{1000}}r\right)^{s+d+h}.
\]
Under \eqref{eq:trace-hypotheses}, $r\ge2L^{1000}$, so the common
ratio is at most $1/2$. The enlarged sum factors into three
convergent geometric series:
\[
 \sum_{s,d,h\ge0}\left(\frac{L^{1000}}r\right)^{s+d+h}
   =\left(\sum_{j\ge0}\left(\frac{L^{1000}}r\right)^j\right)^3
   =\left(1-\frac{L^{1000}}r\right)^{-3}\le8.
\]
Finally
\[
 128pL\le640p^2\le64^p\qquad(p\ge2);
\]
the first inequality uses $L=4p+1\le5p$. The last holds at $p=2$,
where $640p^2=2560\le4096=64^p$; when $p$ increases by one,
the quantity $640p^2$ grows by a factor $((p+1)/p)^2\le9/4<64$.
Therefore
\[
 \bigl|\E\tr A^{2p}\bigr|
    \le128pL(9K_2)^p(\delta\theta)^p
    \le(64\cdot9K_2)^p(\delta\theta)^p
    =K_0^p(\delta\theta)^p,
\]
which is \eqref{eq:trace-bound}.
\end{proof}

\section{Proof of the main theorem}\label{sec:main-proof}

Theorem~\ref{thm:main} requires a uniform sample of a deterministic
size, whereas Lemma~\ref{cor:bernoulli} controls Bernoulli sampling
and applies only above a rank threshold. These two
restrictions are addressed separately. The first subsection transfers the Bernoulli
estimate to a fixed sample size by coupling two nearby densities.
The second supplies a second-moment bound for the remaining ranks.
The final subsection checks these estimates at the prescribed sample
size in \eqref{eq:prescribed-width} for every rank.

Recall that $X=HD_2HD_1V$ has orthonormal columns. By
\eqref{eq:sampled-gram}, the matrix to control is
$(n/k)X^TS_JS_J^TX$, where $J$ is a uniform $k$-element coordinate
sample independent of the sign diagonals. Full sampling gives $I_r$
exactly; the estimates below control the error when fewer coordinates are used.

\subsection{From Bernoulli sampling to a fixed sample size}

The coupling below uses one ordering of independent uniform
variables. Its Bernoulli marginals give upper and lower bounds
for the same fixed-size Gram matrix. Near full sampling, the
orthonormality of $X$ supplies the upper bound directly.
For a coordinate set $B$, write
$E_B=\operatorname{diag}(\one_{\{i\in B\}}:i\in K)$; in particular,
$E_J=S_JS_J^T$. The notation $M\preceq N$ means that $N-M$ is positive
semidefinite, so this order compares every quadratic form.

\begin{lemma}\label{lem:sampling}
Let \(X\) be an orthonormal frame, possibly random and independent
of sampling. For \(0<\varepsilon<1\) and an integer \(1\le k\le n\), put
\[
 \alpha=\varepsilon/4,\qquad
 \theta_-=(1-\alpha)k/n,\qquad q_+=(1+\alpha)k/n.
\]
Let $\gamma\ge0$. Suppose the normalized Bernoulli Gram of density
$\theta_-$ has error at most $\alpha$ except with probability $\gamma$.
If $q_+<1$, set $\theta_+=q_+$ and suppose the same guarantee
holds at density $\theta_+$. Explicitly, the required bounds are
\begin{equation}\label{eq:sampling-hypotheses}
\begin{aligned}
 \Prob\{\norm{\theta_-^{-1}X^TE_-X-I_r}>\alpha\}
          &\le\gamma,\\
 \Prob\{\norm{\theta_+^{-1}X^TE_+X-I_r}>\alpha\}
          &\le\gamma\quad\text{if }q_+<1.
\end{aligned}
\end{equation}
Each $E_\pm$ used here is a Bernoulli diagonal projection of the
indicated density, independent of $X$. These are marginal statements;
$E_-$ and $E_+$ need not be independent of one another.
When $q_+\ge1$, no upper Bernoulli hypothesis is required.
An independent uniform \(k\)-element sample then satisfies
\begin{equation}\label{eq:sampling-bound}
 \Prob\left\{\norm{\frac nk X^TS_JS_J^TX-I_r}>\varepsilon\right\}
 \le2\gamma+2\exp(-\varepsilon^2 k/48).
\end{equation}
\end{lemma}
\begin{proof}
The proof uses a common-ordering realization of the nested coupling between
Bernoulli and fixed-size uniform samples described by H\'ajek
\cite[Section~2, pp.~362--363]{hajek_sampling}.
The lower threshold is always in $(0,1)$; an upper Bernoulli sample
is needed only when $q_+<1$.
Take independent uniform \(U_i\) on \([0,1]\), independent of \(X\).
Let \(J\) index the \(k\) smallest values, and put
\(B_-=\{i:U_i\le\theta_-\}\), \(N_-=|B_-|\).
When $q_+<1$, define $B_+$ and $N_+$ in the same way using $\theta_+$.
There are almost surely no ties. Exchangeability makes every
$k$-element set equally likely to be $J$, and $J$ is independent
of $X$ because it depends only on the $U_i$. For each threshold,
the indicators are independent Bernoulli variables with the required
mean $\theta_-$ or $\theta_+$.

On $N_-\le k$, every value below the lower threshold is among
the first $k$, so $B_-\subseteq J$. Positive semidefinite order gives
\[
 (1-\alpha)\theta_-^{-1}X^TE_{B_-}X
       \preceq \frac nkX^TE_JX,
\]
because $(1-\alpha)\theta_-^{-1}=n/k$. On the lower Gram success
event in \eqref{eq:sampling-hypotheses}, this implies
\[
 (1-\alpha)^2I_r\preceq\frac nkX^TE_JX.
\]

If $q_+<1$, the count event $k\le N_+$ gives $J\subseteq B_+$.
On the upper Gram success event, the corresponding comparison is
\[
 \frac nkX^TE_JX
 \preceq (1+\alpha)\theta_+^{-1}X^TE_{B_+}X
 \preceq (1+\alpha)^2I_r.
\]
If $q_+\ge1$, including equality, $E_J\preceq I_n$ and $X^TX=I_r$
instead give the deterministic bound
\begin{equation}\label{eq:saturated-upper-bound}
 \frac nkX^TE_JX\preceq\frac nkI_r
       \preceq(1+\alpha)I_r\preceq(1+\alpha)^2I_r.
\end{equation}
Here $n/k\le1+\alpha$ is exactly the condition $q_+\ge1$.
Thus, in either case, the successful comparisons put the spectrum
between \((1-\alpha)^2\) and \((1+\alpha)^2\). These endpoints are within
$1\pm\varepsilon$, because $\alpha=\varepsilon/4$ and
$0<\varepsilon<1$ give
\[
 (1-\alpha)^2\ge1-2\alpha\ge1-\varepsilon,
 \qquad
 (1+\alpha)^2=1+\frac\varepsilon2+\frac{\varepsilon^2}{16}
                   \le1+\frac{9\varepsilon}{16}\le1+\varepsilon.
\]
It remains to bound the relevant count failures.

For a binomial \(N\) of positive mean \(\mu\) and $a>0$, the tail bounds are
\begin{equation}\label{eq:binomial-tails}
 \Prob(N-\mu\ge a)\le e^{-a^2/(2\mu+2a/3)},\qquad
 \Prob(N-\mu\le-a)\le e^{-a^2/(2\mu)}.
\end{equation}
Here is the exponential-moment calculation. If
$N\sim\operatorname{Bin}(n,\beta)$, so $\mu=n\beta$, independence
and $1+u\le e^u$ give
\[
 \E e^{t(N-\mu)}
   =e^{-t\mu}(1-\beta+\beta e^t)^n
   \le\exp\bigl(\mu(e^t-1-t)\bigr).
\]
For $0<t<3$, the factorial bound $j!\ge2\cdot3^{j-2}$,
$j\ge2$, gives the termwise estimate
\[
 e^t-1-t=\sum_{j\ge2}\frac{t^j}{j!}
       \le\frac{t^2}{2}\sum_{j\ge0}(t/3)^j
       =\frac{t^2}{2(1-t/3)}.
\]
Markov's inequality consequently yields
\[
 \Prob(N-\mu\ge a)
   \le\exp\left(-ta+\frac{\mu t^2}{2(1-t/3)}\right).
\]
Choose $t=a/(\mu+a/3)$, which lies in $(0,3)$. Substitution makes
the exponent $-a^2/(2\mu+2a/3)$, proving the upper-tail bound.
For the lower tail, apply the same moment bound with $-t$ and use
\[
 e^{-t}-1+t=\int_0^t(1-e^{-u})\,du
          \le\int_0^t u\,du=\frac{t^2}{2}\qquad(t\ge0).
\]
Thus
\[
 \Prob(N-\mu\le-a)\le\exp(-ta+\mu t^2/2).
\]
Taking $t=a/\mu$ gives the lower-tail bound in
\eqref{eq:binomial-tails}.

For $N_-$, use $\mu=(1-\alpha)k$ and $a=\alpha k$; when
$q_+<1$, use $\mu=(1+\alpha)k$ and the same $a$ for $N_+$.
The bounds give
\[
 \begin{aligned}
 \Prob(N_->k)
   &\le\exp\left(-\frac{\alpha^2k}{2(1-\alpha)+2\alpha/3}\right)
      \le e^{-\alpha^2k/2},\\
 \Prob(N_+<k)
   &\le\exp\left(-\frac{\alpha^2k}{2(1+\alpha)}\right)
      \le e^{-\alpha^2k/3}\qquad(q_+<1),
 \end{aligned}
\]
since $\alpha<1/4$. When $q_+<1$, the count failures sum to at most
$2e^{-\alpha^2k/3}$ and the two Gram failures to at most $2\gamma$.
When $q_+\ge1$, only the lower count and lower Gram can fail,
with total probability at most $e^{-\alpha^2k/2}+\gamma$.
The common bound is therefore
\[
 \Prob\left\{\norm{\frac nkX^TE_JX-I_r}>\varepsilon\right\}
       \le2\gamma+2e^{-\alpha^2k/3}
       =2\gamma+2e^{-\varepsilon^2k/48},
\]
which is \eqref{eq:sampling-bound}. Only the marginal Gram
guarantees stated for the relevant densities enter this union bound.
\end{proof}

\subsection{A second-moment estimate for the remaining ranks}

The high-order trace argument is unnecessary for bounded ranks.
The next estimate averages over both sign diagonals and the uniform
sample. Its quadratic dependence on $r$ will be absorbed into the
universal constant once the rank cutoff has been fixed. The
Frobenius norm is estimated because its
expectation can be expanded using pairs of sampled rows. The operator
norm satisfies $\norm{M}\le\|M\|_F$, so this also bounds the
spectral error needed in Theorem~\ref{thm:main}.

\begin{lemma}\label{lem:small}
For every deterministic orthonormal frame \(V\), let \(X=HD_2HD_1V\)
and let $1\le k\le n$ be an integer. Let \(J\) be an independent
uniform \(k\)-element subset. Then
\begin{equation}\label{eq:small-rank-moment}
 \E_{D_1,D_2,J}\left\|\frac nkX^TS_JS_J^TX-I_r\right\|_F^2
                                      \le\frac{r(r+1)}k.
\end{equation}
Consequently, for every $\varepsilon>0$,
\begin{equation}\label{eq:small-rank-tail}
 \Prob\left\{\norm{\frac nkX^TS_JS_J^TX-I_r}>\varepsilon\right\}
       \le\frac{r(r+1)}{k\varepsilon^2}.
\end{equation}
\end{lemma}
\begin{proof}
The case \(n=1\), and every case \(k=n\), is exact.
For \(n>1\), let \(x_i^T\) be the rows of \(X\), and write
\(G=(n/k)X^TS_JS_J^TX\). The expectation is first taken over $J$ with $X$
fixed; the fourth moments of the rows are then bounded by averaging over
the sign diagonals.

A uniform $k$-element sample has inclusion probabilities
\[
 \Prob(i\in J)=\frac kn,\qquad
 \Prob(i,j\in J)=\frac{k(k-1)}{n(n-1)}\quad(i\ne j).
\]
Since $G=(n/k)\sum_{i\in J}x_ix_i^T$ and
$\sum_i x_ix_i^T=X^TX=I_r$, the first identity gives
$\E[G\mid X]=I_r$. Also,
\[
 \sum_{i\ne j}(x_ix_i^T)(x_jx_j^T)
       =I_r-\sum_i\|x_i\|^2x_ix_i^T.
\]
This follows by expanding $(\sum_i x_ix_i^T)^2$ and removing the
terms with $i=j$. Taking traces, the diagonal terms sum to
$\sum_i\|x_i\|^4$, while the off-diagonal terms sum to
$r-\sum_i\|x_i\|^4$. Applying the two inclusion probabilities to
$\tr G^2$ therefore yields
\[
 \begin{aligned}
 \E[\|G-I_r\|_F^2\mid X]
   &=\E[\tr G^2\mid X]-r\\
   &=\left(\frac nk\right)^2
     \left[\frac kn\sum_i\|x_i\|^4
       +\frac{k(k-1)}{n(n-1)}
          \left(r-\sum_i\|x_i\|^4\right)\right]-r.
 \end{aligned}
\]
The first equality uses symmetry of $G$ and $\E[G\mid X]=I_r$.
Collecting the row-moment and constant terms gives
\begin{equation}\label{eq:conditional-sampling-moment}
 \E\!\left[\|G-I_r\|_F^2\mid X\right]
  =\frac{n-k}{k(n-1)}
                \left(n\sum_i\|x_i\|^4-r\right).
\end{equation}

Condition on $D_1$, which fixes $U=HD_1V$, with rows $u_j^T$.
Then $U^TU=I_r$ and $X=HD_2U$. Recalling
$\xi_{2,j}=(D_2)_{jj}$, the $i$th row satisfies
\[
 \sqrt n\,x_i=\sum_j(\sqrt n H_{ij})\xi_{2,j}u_j.
\]
For fixed $i$, the factors $\sqrt n H_{ij}$ are deterministic
signs, so multiplying by them preserves the independent Rademacher
law. Thus each $\sqrt n\,x_i$ has the law of $\sum_j\xi_j u_j$,
where the $\xi_j$ are independent uniform signs.

To compute its fourth moment, expand
\[
 \left\|\sum_j\xi_j u_j\right\|^4
   =\sum_{a,b,c,d}\xi_a\xi_b\xi_c\xi_d
                    \langle u_a,u_b\rangle\langle u_c,u_d\rangle.
\]
A sign product has zero expectation unless each index occurs an
even number of times. The three pairings account for all such
patterns, but count the all-equal pattern three times. Subtracting
it twice gives the exact identity
\[
 \E[\xi_a\xi_b\xi_c\xi_d]
   =\one_{\{a=b\}}\one_{\{c=d\}}
     +\one_{\{a=c\}}\one_{\{b=d\}}
     +\one_{\{a=d\}}\one_{\{b=c\}}
     -2\one_{\{a=b=c=d\}}.
\]
The first pairing contributes $(\sum_j\|u_j\|^2)^2=r^2$.
Each of the other two contributes
\[
 \sum_{a,b}\langle u_a,u_b\rangle^2
    =\tr\bigl((U^TU)^2\bigr)=r.
\]
The all-equal term contributes $\sum_j\|u_j\|^4$. Consequently
\[
 \E_{\xi}\left\|\sum_j\xi_j u_j\right\|^4
             =r^2+2r-2\sum_j\|u_j\|^4\le r^2+2r.
\]
Each row has conditional fourth moment at most $(r^2+2r)/n^2$.
Thus the row-moment
term in \eqref{eq:conditional-sampling-moment} satisfies
\[
 \E_{D_2}\!\left[n\sum_i\|x_i\|^4\,\middle|\,D_1\right]
       \le r^2+2r.
\]
Average \eqref{eq:conditional-sampling-moment} over $D_2$ using
this bound, and then over $D_1$. Since $(n-k)/(n-1)\le1$,
\[
 \E\|G-I_r\|_F^2
    \le\frac{n-k}{k(n-1)}(r^2+2r-r)
    \le\frac{r(r+1)}k,
\]
which proves \eqref{eq:small-rank-moment}.

Finally, norm domination and Markov's inequality give
\[
 \Prob\{\norm{G-I_r}>\varepsilon\}
    \le\frac{\E\|G-I_r\|_F^2}{\varepsilon^2}
    \le\frac{r(r+1)}{k\varepsilon^2}.
\]
This proves \eqref{eq:small-rank-tail} at every sample size $1\le k\le n$.
\end{proof}
\subsection{Choice of moment order and completion}

Small ranks are handled by the second-moment bound in
Lemma~\ref{lem:small}; large ranks are handled by combining
Lemmas~\ref{cor:bernoulli} and~\ref{lem:sampling}.
The width is fixed by \eqref{eq:prescribed-width} throughout;
full sampling is exact whenever that formula gives $k=n$.

The universal constant can be made explicit. With $K_0$ as in
Proposition~\ref{prop:trace}, choose
\begin{equation}\label{eq:universal-constants}
 K_0=576\cdot4^{24},\qquad R_0=2^{20000},\qquad
 C=200R_0+8196K_0.
\end{equation}
These constants establish a bound uniform in all parameters; they
are not intended as practical oversampling prescriptions.

\begin{proof}[Proof of Theorem~\ref{thm:main}]
Fix $n,r,\varepsilon$ as in the statement, the constants in
\eqref{eq:universal-constants}, and
$k=\min\{n,\lceil Cr/\varepsilon^2\rceil\}$.
If $k=n$, then $S_JS_J^T=I_n$, and orthonormality of $X$ makes
\eqref{eq:sampled-gram} equal to $I_r$ for every draw.
This also covers $n=1$. Henceforth $k<n$, so
\begin{equation}\label{eq:nonfull-width}
 k=\lceil Cr/\varepsilon^2\rceil\ge Cr/\varepsilon^2.
\end{equation}
In particular, $k\ge r$ and $r/n<\varepsilon^2/C<1/2$.
Write $G=(n/k)X^TS_JS_J^TX$.

Suppose first that $r<R_0$. Apply Lemma~\ref{lem:small} at this
same $k$. Its tail bound gives
\[
 \Prob\{\norm{G-I_r}>\varepsilon\}
 \le\frac{r(r+1)}{k\varepsilon^2}
 \le\frac{r+1}{C}
 \le\frac{2R_0}{200R_0}=0.01.
\]
Here $r+1\le2R_0$ and $C\ge200R_0$ suffice; the estimate holds
at the prescribed width without any further choice of sample size.

Now suppose $r\ge R_0$. The first step is to verify that the
integer \(p=\lceil\log_2(3000r)\rceil\) satisfies the rank condition
of Lemma~\ref{cor:bernoulli}. Set $x=\log_2r\ge20000$.
Since $3000<2^{12}$ and taking the ceiling increases a number by
less than one, $p\le x+13$, so $4p+1\le4x+53$.
The desired comparison follows if the function
\[
 f(x)=x-1-1000\log_2(4x+53)
\]
is positive. At $x=20000$, the inequality $80053<2^{17}$ gives
\[
 f(20000)>20000-1-1000\cdot17=2999>0.
\]
For all $x\ge20000$ it is increasing:
\[
 f'(x)=1-\frac{4000}{(4x+53)\log2}>0,
\]
because $\log2>1/2$ makes the denominator greater than
$80053/2>4000$. Exponentiating the inequality $f(x)>0$ now gives
\[
 r=2^x>2(4x+53)^{1000}\ge2(4p+1)^{1000}.
\]
Thus the rank condition holds throughout this regime.

Set $\alpha=\varepsilon/4$, $\theta_-=(1-\alpha)k/n$, and
$q_+=(1+\alpha)k/n$, as in Lemma~\ref{lem:sampling}.
The lower density belongs to $(0,1)$. Whenever $q_+<1$, the
upper density $\theta_+=q_+$ also belongs to $(0,1)$.
Every Bernoulli density needed in that lemma has effective width at least
\begin{equation}\label{eq:bernoulli-widths}
 n\theta_-=(1-\alpha)k\ge\frac34 k
 \ge\frac{1536K_0r}{\varepsilon^2}
 \ge\frac{64K_0r}{(\varepsilon/4)^2}.
\end{equation}
The second inequality involving $K_0$ follows from
\eqref{eq:nonfull-width} and $C\ge2048K_0$; the last threshold is
$1024K_0r/\varepsilon^2$. Thus the moment order, rank condition,
density range, and effective-width condition of
Lemma~\ref{cor:bernoulli} all hold with $\eta=\alpha$.
Each required Bernoulli failure probability is at most $1/1000$,
including when the density exceeds $1/2$.
If $q_+\ge1$, the upper spectral edge is instead controlled by
\eqref{eq:saturated-upper-bound}, including the boundary $q_+=1$.
Lemma~\ref{lem:sampling} therefore applies in both cases with
$\gamma=1/1000$ and gives
\begin{equation}\label{eq:final-failure}
 \Prob\{\norm{G-I_r}>\varepsilon\}
 \le\frac2{1000}+2e^{-\varepsilon^2k/48}<0.01.
\end{equation}
Indeed, \eqref{eq:nonfull-width} gives
$\varepsilon^2k/48\ge2048K_0r/48>42$, and
$e^{42}\ge1+42+42^2/2=925>250$. Hence the right side is strictly
below $2/1000+2/250=0.01$.

Every case uses the same prescribed $k$ and the same constant $C$.
These depend only on $n,r,\varepsilon$ and satisfy $r\le k\le n$.
The failure bound holds for every deterministic orthonormal frame $V$
fixed before the signs and sample, so taking the supremum over those
frames preserves it. Finally, \eqref{eq:sampled-gram} identifies
$G$ with the Gram matrix in \eqref{eq:main-bound}, completing the proof.
\end{proof}

\section{Concluding remarks}\label{sec:conclusion}

The quantitative estimates here prioritize a bound uniform in all
parameters. It remains to determine whether sharper trace and counting
estimates yield oversampling constants useful in practice. A separate
structural question is how far the argument extends beyond the Walsh
model. Within the present proof strategy, the ingredients to replace
are the joint entry-cumulant estimates and the character constraints
used here.

\section*{Formalization and verification}
A Lean~4 formalization of the prescribed-width theorem and its
supporting results is available in the
\href{https://github.com/yuningyang19/OpenProblemsInNLA_TR-01/tree/ed21181197ac839eac95f549404f94e7e3aa6e10/lean}{current companion repository},
together with manuscript-to-Lean correspondence records and
reproduction instructions. On 12 September 2026, the resolution was
incorporated into the Open Problems in Numerical Linear Algebra
collection, where
\href{https://github.com/ajt60gaibb/OpenProblemsInNLA/blob/5adea969c17391693978ada2674d25bb5c3daeb1/randomized-and-low-rank-approximation/TR-01/README.md}{TR-01 is listed}
as ``Lean verified---complete affirmative resolution.'' The entry
links mathematical, statement-comparison, and Lean-verification
reviews conducted by independent Codex agents.

\begin{small}
\section*{Acknowledgments}
The author thanks Prof.~Alex Townsend for maintaining the Open Problems
in Numerical Linear Algebra collection and for incorporating the
resolution of TR-01 into the repository.
ChatGPT 6.0 Astra was used throughout this project for exploration of
proof strategies, development and refinement of mathematical arguments,
manuscript preparation, and Lean 4 formalization.
\end{small}

\end{document}

%% file: figures/p7_contraction.tex
\begin{figure}[htbp]
\centering
\begin{minipage}[t]{0.47\textwidth}
\centering
\begin{tikzpicture}[
  vertex/.style={circle,draw,fill=white,minimum size=13pt,inner sep=1pt},
  removed/.style={vertex,fill=black!8},
  doubled/.style={draw=blue!65!black,line width=0.5pt,
                  double=white,double distance=1.4pt},
  bundled/.style={draw=black!55,line width=1.4pt},
  every node/.style={font=\small}]
\node[vertex] (b) at (0,0) {$b$};
\node[removed] (x) at (-0.7,1.4) {$x$};
\node[removed] (y) at (0.9,1.4) {$y$};
\node[vertex] (c) at (2.7,0) {$c$};
\node[removed] (z) at (1.35,-1.1) {$z$};
\draw[doubled] (b) -- (x) -- (y) -- (b);
\draw[doubled] (b) -- (z) -- (c);
\draw[bundled] (b) -- node[above=3pt,text=black] {$4$} (c);
\path[use as bounding box] (-1.1,-1.5) rectangle (3.1,2.0);
\end{tikzpicture}

\small (a) Before suppression
\end{minipage}\hfill
\begin{minipage}[t]{0.50\textwidth}
\centering
\begin{tikzpicture}[
  vertex/.style={circle,draw,fill=white,minimum size=13pt,inner sep=1pt},
  doubled/.style={draw=blue!65!black,line width=0.5pt,
                  double=white,double distance=1.4pt},
  bundled/.style={draw=black!55,line width=1.4pt},
  every node/.style={font=\small}]
\node[vertex] (b) at (0,0) {$b$};
\node[vertex] (c) at (2.7,0) {$c$};
\coordinate (leftport) at (-0.16,0.16);
\coordinate (rightport) at (0.16,0.16);
\draw[doubled] (leftport)
  .. controls (-1.05,2.0) and (1.05,2.0) ..
  node[pos=0.5,above=4pt] {length $2$} (rightport);
\draw[doubled] (b) to[bend right=43]
  node[midway,below=5pt] {length $1$} (c);
\draw[bundled] (b) -- node[pos=0.8,above=3pt,text=black] {$4$} (c);
\fill[blue!65!black] (leftport) circle (1.25pt);
\fill[blue!65!black] (rightport) circle (1.25pt);
\draw[thin] (leftport) -- (-0.48,0.48)
  node[anchor=east,font=\scriptsize] {$\{b_1,b_2\}$};
\draw[thin] (rightport) -- (0.48,0.48)
  node[anchor=west,font=\scriptsize] {$\{b_3,b_4\}$};
\path[use as bounding box] (-1.7,-1.5) rectangle (3.1,2.0);
\end{tikzpicture}

\small (b) After suppression
\end{minipage}
\caption{The graph of \eqref{eq:encoding-example-word} and its contraction.
Each blue double stroke represents two parallel edge occurrences;
each gray segment labeled $4$ represents four direct $b$--$c$ edges.
Suppressing $x,y$ gives the returning link of length two, with its two
ports marked at $b$. Suppressing $z$ gives the parallel link of length one.
Lengths count removed internal vertices.}
\label{fig:p7-contraction}
\end{figure}
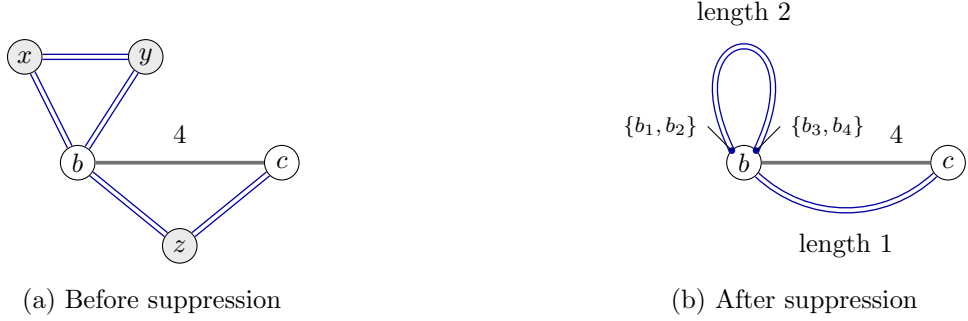